\documentclass[12pt]{amsart}
\usepackage{amssymb}
\usepackage{verbatim}
\usepackage{hyperref}
\usepackage{color}

\begin{document}

\newtheorem{theorem}{Theorem}    
\newtheorem{proposition}[theorem]{Proposition}
\newtheorem{conjecture}[theorem]{Conjecture}
\def\theconjecture{\unskip}
\newtheorem{corollary}[theorem]{Corollary}
\newtheorem{lemma}[theorem]{Lemma}
\newtheorem{sublemma}[theorem]{Sublemma}
\newtheorem{observation}[theorem]{Observation}
\theoremstyle{definition}
\newtheorem{definition}{Definition}
\newtheorem{notation}[definition]{Notation}
\newtheorem{remark}[definition]{Remark}
\newtheorem{question}[definition]{Question}
\newtheorem{questions}[definition]{Questions}
\newtheorem{example}[definition]{Example}
\newtheorem{problem}[definition]{Problem}
\newtheorem{exercise}[definition]{Exercise}

\numberwithin{theorem}{section} \numberwithin{definition}{section}
\numberwithin{equation}{section}

\def\earrow{{\mathbf e}}
\def\rarrow{{\mathbf r}}
\def\uarrow{{\mathbf u}}
\def\varrow{{\mathbf V}}
\def\tpar{T_{\rm par}}
\def\apar{A_{\rm par}}

\def\reals{{\mathbb R}}
\def\torus{{\mathbb T}}
\def\heis{{\mathbb H}}
\def\integers{{\mathbb Z}}
\def\naturals{{\mathbb N}}
\def\complex{{\mathbb C}\/}
\def\distance{\operatorname{distance}\,}
\def\support{\operatorname{support}\,}
\def\dist{\operatorname{dist}\,}
\def\Span{\operatorname{span}\,}
\def\degree{\operatorname{degree}\,}
\def\kernel{\operatorname{kernel}\,}
\def\dim{\operatorname{dim}\,}
\def\codim{\operatorname{codim}}
\def\trace{\operatorname{trace\,}}
\def\Span{\operatorname{span}\,}
\def\dimension{\operatorname{dimension}\,}
\def\codimension{\operatorname{codimension}\,}
\def\nullspace{\scriptk}
\def\kernel{\operatorname{Ker}}
\def\ZZ{ {\mathbb Z} }
\def\p{\partial}
\def\rp{{ ^{-1} }}
\def\Re{\operatorname{Re\,} }
\def\Im{\operatorname{Im\,} }
\def\ov{\overline}
\def\eps{\varepsilon}
\def\lt{L^2}
\def\diver{\operatorname{div}}
\def\curl{\operatorname{curl}}
\def\etta{\eta}
\newcommand{\norm}[1]{ \|  #1 \|}
\def\expect{\mathbb E}
\def\bull{$\bullet$\ }
\def\C{\mathbb{C}}
\def\R{\mathbb{R}}
\def\Rn{{\mathbb{R}^n}}
\def\Sn{{{S}^{n-1}}}
\def\M{\mathbb{M}}
\def\N{\mathbb{N}}
\def\Q{{\mathbb{Q}}}
\def\Z{\mathbb{Z}}
\def\F{\mathcal{F}}
\def\L{\mathcal{L}}
\def\S{\mathcal{S}}
\def\supp{\operatorname{supp}}
\def\dist{\operatorname{dist}}
\def\essi{\operatornamewithlimits{ess\,inf}}
\def\esss{\operatornamewithlimits{ess\,sup}}
\def\xone{x_1}
\def\xtwo{x_2}
\def\xq{x_2+x_1^2}
\newcommand{\abr}[1]{ \langle  #1 \rangle}

\newcommand{\Norm}[1]{ \left\|  #1 \right\| }
\newcommand{\set}[1]{ \left\{ #1 \right\} }
\def\one{\mathbf 1}
\def\whole{\mathbf V}
\newcommand{\modulo}[2]{[#1]_{#2}}

\def\scriptf{{\mathcal F}}
\def\scriptg{{\mathcal G}}
\def\scriptm{{\mathcal M}}
\def\scriptb{{\mathcal B}}
\def\scriptc{{\mathcal C}}
\def\scriptt{{\mathcal T}}
\def\scripti{{\mathcal I}}
\def\scripte{{\mathcal E}}
\def\scriptv{{\mathcal V}}
\def\scriptw{{\mathcal W}}
\def\scriptu{{\mathcal U}}
\def\scriptS{{\mathcal S}}
\def\scripta{{\mathcal A}}
\def\scriptr{{\mathcal R}}
\def\scripto{{\mathcal O}}
\def\scripth{{\mathcal H}}
\def\scriptd{{\mathcal D}}
\def\scriptl{{\mathcal L}}
\def\scriptn{{\mathcal N}}
\def\scriptp{{\mathcal P}}
\def\scriptk{{\mathcal K}}
\def\frakv{{\mathfrak V}}

\title[Multilinear Marcinkiewicz integrals on Companato spaces]
{The existence and boundedness of multilinear Marcinkiewicz integrals on Companato spaces}
\author{Qingying Xue}
\address{
        Qingying Xue\\
        School of Mathematical Sciences\\
        Beijing Normal University \\
        Laboratory of Mathematics and Complex Systems\\
        Ministry of Education\\
        Beijing 100875\\
        People's Republic of China}
\email{qyxue@bnu.edu.cn}
\author{K\^{o}z\^{o} Yabuta}
\address{K\^{o}z\^{o} Yabuta\\Research center for Mathematical
Science \\Kwansei Gakuin University\\Gakuen 2-1, Sanda 669-1337\\
Japan }
\thanks{The second author was supported partly by NSFC
(No. 11471041), the Fundamental Research Funds for the Central Universities (No. 2014KJJCA10) and NCET-13-0065. The third named author was supported partly by Grant-in-Aid for Scientific Research (C) Nr. 23540228, Japan Society
for the Promotion of Science.\\ \indent Corresponding
author: Qingying Xue\indent Email: qyxue@bnu.edu.cn}
%

\maketitle
\begin{abstract}In this paper, we established the boundedness of m-linear
Marcinkiewicz integral on Campanato type spaces. We showed that
if the $m$-linear Marcinkiewicz integral is finite for one point,
then it is finite almost everywhere. Moreover, the following norm inequality holds,
$$\|\mu(\vec{f})\|_{\mathcal{E}^{\alpha,p}}
\leq C\prod_{j=1}^m\|f_j\|_{\mathcal{E}^{\alpha_j,p_j}},$$
where $\mathcal{E}^{\alpha,p}$ is the classical Campanato spaces.
\end{abstract}
\section{Introduction and Main Results}

It is well known that the following classical Marcinkiewicz integral of higher
dimension was first introduced and studied by E.~M.~Stein \cite{Stein1}
in 1958.
\begin{equation}\label{def:fracMar-1}
\mu_{\Omega}(f)(x)
=\bigl(\int_{0}^{\infty}\bigl|
\int_{|y|\le t}f(x-y)\frac{\Omega(y)}{|y|^{n-1}}dy\bigr|^2\frac{dt}{t^3}
\bigr)^{\frac{1}{2}},
\end{equation}
Stein showed that the Marcinkiewicz integral $\mu_{\Omega}$
is of weak type $(1, 1)$ and type $(p, p)$ $(1 < p \leqslant
2)$, where the Lipschitz continuous function $\Omega$ is homogeneous
of degree zero and its integration on the unit sphere vanishes. Later,
Stein's $L^p$ result was further extended by Benedek, Calder\'{o}n and
Panzone \cite{BCP} to the case $1<p<\infty$ when the kernel belongs to
$C^1(S^{n-1})$.
In the connection of $\mu_{\Omega}$, the parametric Marcinkiewicz integral
operator $\mu_{\Omega,\rho}$ was first considered by H\"ormander
\cite{Hormander1} in 1960. Since then, many works have been done for
Marcinkiewicz integral or its related parametric operators.
$L^p$ boundedness for these operators were well discussed \cite{DFP1,DFP2,DLS}
and there were also some related results on function spaces, such as
Triebel-Lizorkin spaces $\dot F_{pq}^{\alpha}(\Rn)$ \cite{CFY, CZ, XQY},
Hardy spaces \cite{DLX2, SiWangJiang}, Campanato spaces,
\cite{W, Ws, WC, DLX}.
A nice survey was given by Lu \cite{Lu1}.

It is also well known that, the multilinear operators were first introduced and
studied by Coifman and Meyer \cite{CM1, CM2} in the 70s. After the celebrated
works of them, the topic was retaken by several authors, including Christ
and Journ\'{e} \cite{Crist}, Kenig and Stein \cite{KS}, Grafakos and Torres
\cite{GT, multi C-Z} and Lerner et al \cite{LPO}.
The study of multilinearization of
Littlewood-Paley's square function can be traced back to the work of Coifman
and Meyer \cite{CM3}. Some improvement can be found in the works of
Yabuta \cite{Y2}, Sato and Yabuta \cite{SY}. Recently, some weighted results
for multilinear Littlewood-Paley operators, in particular, the multilinear
Marcinkiewicz integral, were established in \cite{CX}.
To state some known results, we first introduce some definitions.

\begin{definition}\label{multilinear Marcinkiewicz integral}
(Multilinear Marcinkiewicz integral \cite{CX}).
Let $\Omega$ be a function defined on $(\R^n)^m$ with the following properties:
Let $\Omega$ be a function defined on ${(\mathbb{R}^n)}^m$ with the following
properties:
\begin{enumerate}\renewcommand{\labelenumi}{$($\roman{enumi}$)$}
\item $\Omega$ is homogeneous of degree $0$, i.e.,
\begin{equation}\label{homogeneous}
\Omega(\lambda \vec{y}) = \Omega(\vec{y});\text{\ \ \quad for any }
\lambda > 0  \text{ and}\ \ \vec{y}=(y_1,\cdots, y_m)\in(\mathbb{R}^n)^{m}.
\end{equation}
\item $\Omega$ is Lipschitz continuous on $({\mathbb{S}}^{n-1})^m$, i.e.
there are $0 < \alpha < 1$ and $C> 0$ such that for any
$\xi=(\xi_1,\cdots \xi_m),\eta=(\eta_1, \cdots, \eta_m) \in(\mathbb{R}^n)^{m}$
\begin{equation}
{|\Omega(\xi) - \Omega(\eta)|}\label{lip}  \le  C \, {|\xi' -
\eta'|}^\alpha,
\end{equation}
where $(y_1,\cdots, y_m)'=\frac{(y_1, \cdots, y_m)}{|y_1| + \cdots + |y_m|}$,
and it should be noted that $(y_1, \cdots, y_m)'$ is not an element of
$ ({\mathbb{S}}^{n -1})^m$;
 \item The integration of $\Omega$ on each unit sphere vanishes,
\begin{equation}\label{vanish}
\int_{S^{n-1} }\Omega(y_1,\cdots,y_m) \; dy_j=0,\quad j=1,\cdots,m.
\end{equation}
\end{enumerate}
For any $\vec{f}=(f_1,\cdots,f_m)\in S\times \cdots \times S$, we can define
the operator $F_t$ for any $t>0$ as
\begin{align}\label{F_t}
F_t(\vec{f})(x)
&=\frac{\chi_{(B(0,t))^m} \Omega(\vec{\cdot})}{t^m|\vec{\cdot}|^{m(n-1)}}
\ast(f_1\otimes \cdots \otimes f_m)(x)
\\
&=\frac{1}{t^m}\int_{(B(0,t))^m} \frac{\Omega(\vec{y})}{|\vec{y}|^{m(n-1)}}
\prod_{i=1}^m f_i(x-y_i)\; d\vec{y},\nonumber
\end{align}
where $|\vec{y}|=|y_1|+\cdots+|y_m|$ and $B(x,t)=\{y\in R^n:|y-x|\leq t\}$.
Finally, the multilinear Marcinkiewicz integral $\mu$ is defined by
\begin{align}\label{Marcinkiewicz integral}
\mu(\vec{f})(x)=\left(\int_0^\infty |F_t(\vec{f})(x)|^2\;
\frac{dt}{t}\right)^{1/2}.
\end{align}
\end{definition}
If $m=1$, it is easy to see that $\mu(\vec{f})$ coincides with
$\mu_{\Omega}({f})$ which was defined and studied by Stein \cite{Stein1}.
If $m\ge 2$, Chen, Xue and Yabuta \cite{CX} recently gave the following result.
\\
\noindent\textbf{Theorem A} (Estimate for $\mu$)  {\rm (\cite{CX}).}
Suppose $\mu$ is bounded from $L^{q_1}\times\cdots\times L^{q_m}$ to $L^q$
for some $1<q_1,\dots, q_m<\infty$ with $1/q=1/q_1+\cdots+1/q_m$. Then for
 $1<p_1,\cdots,p_m<\infty$ with
$\frac{1}{p}=\frac{1}{p_1}+\cdots+\frac{1}{p_m}$, there is a $C>0$ such that
$$
\|{{\mu}}(\vec{f})\|_{L^p}\leq C\prod_{i=1}^m\|f_i\|_{L^{p_i}}.
$$
\par\smallskip
To state some other results, we begin with the definition
 of Campanato type spaces.
\begin{definition} \label{Campanato space}(\textbf{Campanato space}).
(\cite{DLX})
Let $1\leq p<\infty$ and $-n/p\leq \alpha<1$. A locally integrable function
$f$ is said to belong to the Campanato space
$\mathcal{E}^{\alpha,p}(\mathbb{R}^n)$ if there exists a constant $C>0$
 such that for any ball $B\subset \mathbb{R}^n$,
\begin{equation}\|f\|_{\mathcal{E}^{\alpha,p}}:=\sup_B\frac 1{|B|^{\alpha/n}}
\bigg(\frac 1{|B|}\int_B|f(x)-f_B|^pdx\bigg)^{1/p}<\infty,
\end{equation}
where the supremum is taken over all cubes in $\Bbb R^n$ with
sides parallel to the coordinate axes, and $f_B$ denotes the
average of $f$ over $B,$ that is, $f_B=\frac 1{|B|}\int_B f(x)dx$.
\end{definition}
\begin{remark}
It is well known that if $\alpha\in(0,1)$ and $p\in[1,\infty)$, then
$\mathcal{E}^{\alpha,p}(\R^n)=Lip_\alpha(\R^n),$
with equivalent norms; if $\alpha=0$, then $\mathcal{E}^{\alpha,p}(\R^n)$
coincides with $BMO(\R^n)$; and if $\alpha\in(-n/p,0)$,
then $\mathcal{E}^{\alpha,p}(\R^n)$ coincides with the Morrey space
$L^{p,\alpha p+n}(\R^n)$.
\end{remark}
In 1987, Han \cite{Han} proved the following result, which was improved later by Lu, Ding and Xue \cite {DLX} with more weaker conditions assumed on the kernel.

\par\smallskip\noindent
{\bf Theorem B.} \label{thmD}{\it Suppose $\Omega$ is continuous on
$\mathcal {S}^{n-1}$, satisfies a $\text{Lip}_{\alpha}$ condition for
$0<\alpha\le 1$, and its integration on $\mathcal {S}^{n-1}$ vanishes.
If $f\in \textit{BMO}(\mathbb R^n)$ and $\mu_{\Omega}(f)(x)$ is finite on
a set of positive measure, then $\mu_{\Omega}(f)(x)$ is finite a.e.
on $\mathbb R^n$, and
there exists a positive constant $C$, independent of $f$, such that
$$
\|\mu_{\Omega}(f)\|_{\textit{BMO}(\mathbb R^n)}
\leq C\|f\|_{\textit{BMO}(\mathbb R^n)}.
$$
}\par\smallskip
\begin{remark}Previous works can be traced back to the celebrated theorem
given by Wang \cite{W}
and Kurtz \cite{K} for $g$-function. Theorem A also holds for Lusin's Area
integral, Littlewood-Paley
$g_\lambda^*$-function. Moreover, similar results
were treated for the above operators on Campanato type spaces (\cite{Qiu},
\cite{SakamotoY}, \cite{DLX}, \cite{HXMY}).\end{remark}
\par
In 1990, Wang and Chen \cite{WC} gave the following interesting result.
\par\smallskip\noindent
{\bf Theorem C.} \label{thmC}{\it
If $f\in \textit{BMO}(\mathbb R^n)$ and $\mu_{\Omega}(f)(x)$ is finite for
a point $x_0\in\mathbb R^n$, then $\mu_{\Omega}(f)(x)$ is finite a.e. on
$\mathbb R^n$, and
there exists a positive constant $C$, independent of $f$, such that
$\|\mu_{\Omega}(f)\|_{\textit{BMO}(\mathbb R^n)}
\leq C\|f\|_{\textit{BMO}(\mathbb R^n)}.
$}\par\smallskip
Similarly results still hold for Littlewood-Paley $g$-function,
Lusin's area integral $S$,
Littlewood-Paley $g_\lambda^*$-function.
\begin{remark} In 2003, Yabuta \cite{Y3} extended Theorem C to the case of
Campanato type spaces, also for $g_\lambda^*$-function
and Marcinkiewicz integral. In addition, in 2004, Sun \cite{S} got
similar results for $g$-function and Lusin's area integral.
\end{remark}
In this paper, we shall give a multilinear analogue of Theorem C for the
multilinear Marcinkiewicz integral.
Our main results are as follows.
\begin{theorem}\label{existence on BMO}Let $\Omega$ be a function defined on
$\mathbb{R}^{mn}$, satisfying (\ref{homogeneous}), Lipschitz continuous
condition (\ref{lip}) and (\ref{vanish}).
Suppose that $\mu$ is bounded from $L^{q_1}\times\cdots\times L^{q_m}$
to $L^q$ for some $1<q_1,\dots, q_m<\infty$ with $1/q=1/q_1+\cdots+1/q_m$.
For $f_i\in \text{BMO}(\R^n)$,
if ${{\mu}}(\vec{f})(x)$ is finite for a point
$x_0\in\mathbb{R}^n$, then
${{\mu}}(\vec{f})(x)<\infty$ a.e. on $\mathbb{R}^n$, and there exists
a positive constant $C$, independent of $\vec{f}$, such that
$$
\|{{\mu}}(\vec{f})\|_{\mathcal{\text{BMO}}}
\leq C\prod_{i=1}^m\|f_i\|_{\text{BMO}}.
$$
\end{theorem}
\begin{theorem}\label{existence on Campanato space}
Let $\Omega$ be a function defined on $\mathbb{R}^{mn}$, satisfying
(\ref{homogeneous}), Lipschitz continuous condition (\ref{lip}) and
(\ref{vanish}) with index replaced by $\beta$.
Suppose that $\mu$ is bounded from $L^{q_1}\times\cdots\times L^{q_m}$
to $L^q$ for some $1<q_1,\dots, q_m<\infty$ with $1/q=1/q_1+\cdots+1/q_m$.
Suppose that $-\infty<\alpha=\alpha_1+\cdots+\alpha_m<0$ with
$\alpha_1,\cdots, \alpha_m<0$ and $n<p<\infty$, or $-m<\alpha<0$ and
$1<p<\infty$, then for $f_j\in\mathcal{E}^{\alpha_j,p_j}(\R^n)$, and
$1<p_1,\cdots,p_m<\infty$ with
$\frac{1}{p}=\frac{1}{p_1}+\cdots+\frac{1}{p_m}$, $\mu(\vec{f})$ is either
infinite everywhere or finite almost everywhere, and in the latter case,
the following inequality holds,
$$
\|\mu(\vec{f})\|_{\mathcal{E}^{\alpha,p}}
\leq C\prod_{j=1}^m\|f_j\|_{\mathcal{E}^{\alpha_j,p_j}}.
$$
\end{theorem}
\begin{theorem}\label{existence on Lipschitz space}
Let $\Omega$ be the same as in the above theorem.
Suppose that $0<\alpha <\beta/2$, and $f_j\in Lip_{\alpha_j},\,j=1,\cdots,m$.
Then $\mu(\vec{f})$ is either infinite everywhere or  finite almost
everywhere and in the latter case,
$$
\|\mu(\vec{f})\|_{Lip_\alpha}\leq C\prod_{j=1}^m\|f_j\|_{Lip_{\alpha_j}},
$$
where $C>0$ is independent of $\vec{f}$.
\end{theorem}
The article is organized as follows. Some basic lemmas will be presented in Section \ref{Sec-2}. The proof of Theorem \ref{existence on BMO} will be given in Section \ref{proof 1.1}. In Section \ref{Sec-4}, we will demonstrate  the proofs of Theorem \ref{existence on Campanato space} and Theorem \ref{existence on Lipschitz space}. Finally, in Section \ref{Sec-5}, some extensions will be presented for the operators with separated kernels and more general type kernels . 
\section{Some basic lemmas}\label{Sec-2}

To prove our Theorems, we prepare some lemmas.
\begin{lemma}\label{lem:BMO-decreasing}
Let $1\le p<\infty$, $B=B(x_0,r)$, $x\in B$ and $t>8r>0$.
Then for $0\le k\le k_0$ with
$k_0\in\N$ satisfying $2r\le 2^{-k_0}t<4r$ we have
\begin{align*}
\biggl(\frac{1}{|B(x,2^{-k}t)|}\int_{B(x,2^{-k}t)}|f(y)-f_{B}|\,dy
\biggr)^{\frac1p}
&\le c_n\Bigl(\sum_{j=k}^{k_0}2^{-j\alpha}\Bigr)t^\alpha
\|f\|_{\mathcal {E}^{p,\alpha}}
\\
&\le\begin{cases}
cr^\alpha\|f\|_{\mathcal {E}^{p,\alpha}} &\alpha<0,\\
c\log\frac{t}{r}\|f\|_{\mathcal {E}^{p,\alpha}} &\alpha=0,\\
ct^\alpha\|f\|_{\mathcal {E}^{p,\alpha}} &\alpha>0,
\end{cases}
\end{align*}
\end{lemma}
\begin{proof}
We see easily that
\begin{align*}
&\biggl(\frac{1}{|B(x,2^{-k}t)|}\int_{B(x,2^{-k}t)}|f(y)-f_{B}|dy
\biggr)^{\frac1p}
\\
&\le\biggl(\frac{1}{|B(x,2^{-k}t)|}\int_{B(x,2^{-k}t)}|f(y)-f_{B(x,2^{-k}t)}|dy
\biggr)^{\frac1p}
\\
&\ +|f_{B(x,2^{-k}t)}-f_{B(x,2^{-k-1}t)}|+\cdots
+|f_{B(x,2^{-k_0+1}t)}-f_{B(x,2^{-k_0}t)}|+|f_{B(x,2^{-k_0}t)}-f_{B}|.
\end{align*}
The first term is bounded by $(2^{-k}t)^\alpha\|f\|_{\mathcal {E}^{p,\alpha}}$,
and
\begin{align*}
|f_{B(x,2^{-j}t)}-f_{B(x,2^{-j-1}t)}|
&\le
\frac{|B(x,2^{-j}t)|}{|B(x,2^{-j-1}t)|}\frac{1}{|B(x,2^{-j}t)|}
\int_{B(x,2^{-j-1}t)}|f(y)-f_{B(x,2^{-j}t)}|dy
\\
&\le c_n2^n\frac{1}{|B(x,2^{-j}t)|}
\int_{B(x,2^{-j}t)}|f(y)-f_{B(x,2^{-j}t)}|dy
\\
&\le c_n2^n(2^{-j}t)^\alpha\|f\|_{\mathcal {E}^{p,\alpha}}
\end{align*}
for $j=k,\dots,k_0$. For the last term, we have
\begin{align*}
|f_{B(x,2^{-k_0}t)}-f_{B}|
&\le \frac{1}{|B|}\int_{B}|f(y)-f_{B(x,2^{-k_0}t)}|dy
\\
&\le \frac{|B(x,2^{-k_0}t)|}{|B|}\frac{1}{|B(x,2^{-k_0}t)|}
\int_{B(x,2^{-k_0}t)}|f(y)-f_{B(x,2^{-k_0}t)}|dy
\\
&\le c_n4^n(2^{-k_0}t)^\alpha\|f\|_{\mathcal {E}^{p,\alpha}}.
\end{align*}
Altogether, we obtain
\begin{equation*}
\frac{1}{|B(x,2^{-k}t)|}\int_{B(x,2^{-k}t)}|f(y)-f_{B}|dy
\le c_n\Bigl(\sum_{j=k}^{k_0}2^{-j\alpha}\Bigr)\|f\|_{\mathcal {E}^{p,\alpha}}.
\end{equation*}
Since $2r\le 2^{-k_0}t<4r$, we have the required estimate.
\end{proof}

\begin{lemma}\label{lem:BMO-increasing}
Let $1\le p<\infty$, $B=B(x_0,r)$, $k\in\N$, and $\alpha\in\R$. Then for
$x\in B$
\begin{align*}
\biggl(\frac{1}{|B(x,2^{k}r)|}\int_{B(x,2^{k}r)}|f(y)-f_{B}|dy\biggr)^{\frac1p}
&\le c_n\Bigl(\sum_{\ell=0}^{k}2^{\alpha\ell}\Bigr)r^\alpha
\|f\|_{\mathcal {E}^{p,\alpha}(\Rn)}
\\
&\le
\begin{cases}
c_0r^{\alpha}\|f\|_{\mathcal {E}^{p,\alpha}(\Rn)}, &\alpha<0,\\
c_1k\|f\|_{\mathcal {E}^{p,\alpha}(\Rn)}, &\alpha=0,\\
c_2(2^kr)^{\alpha}\|f\|_{\mathcal {E}^{p,\alpha}(\Rn)}, & \alpha>0.
\end{cases}.
\end{align*}
\end{lemma}
\begin{proof}
We see that
\begin{align*}
|f_{B(x,2^{\ell+1}r)}-f_{B(x,{2^\ell}r)}|
&\le \frac{|B(x,2^{\ell+1}r)|^{1+\frac{\alpha}{n}}}{|B(x,2^{\ell}r)|}
\frac{1}{|B(x,2^{\ell+1}r)|^{1+\frac{\alpha}{n}}}\int_{B(x,2^{\ell+1}r)}
|f(y)-f_{B(x,2^{\ell+1}r)}|dy
\\
&\le 2^nc_n(2^{\ell+1}r)^\alpha\|f\|_{\mathcal {E}^{\alpha,1}(\Rn)}.
\end{align*}
Hence we have
\begin{align*}
&\biggl(\frac{1}{|B(x,2^{k}r)|}\int_{B(x,2^{k}r)}|f(y)-f_{B}|dy\biggr)^{\frac1p}
\\
&\le\biggl(\frac{1}{|B(x,2^{k}r)|}\int_{B(x,2^{k}r)}|f(y)-f_{B(x,2^{k}r)}|dy
\biggr)^{\frac1p}
\\
&\ \ +|f_{B(x,2^{k}r)}-f_{B(x,2^{k-1}r)}|+|f_{B(x,2^{k-1}r)}-f_{B(x,2^{k-2}r)}|
+\cdots+|f_{B(x,2r)}-f_{B}|
\\
&\le |B(x,2^{k}r)|^{\frac{\alpha}{n}}\|f\|_{\mathcal {E}^{p,\alpha}(\Rn)}
+2^nc_n(2^{k}r)^\alpha\|f\|_{\mathcal {E}^{\alpha,1}(\Rn)}
+\dots+2^nc_nr^\alpha\|f\|_{\mathcal {E}^{\alpha,1}(\Rn)}
\\
&\le c\Bigl(\sum_{\ell=0}^{k}2^{\alpha\ell}\Bigr)r^\alpha
\|f\|_{\mathcal {E}^{p,\alpha}(\Rn)}.
\end{align*}
from which it follows the desired estimates.
\end{proof}

\begin{lemma}\label{lem:BMO-n-1}
Let $B=B(x_0,r)$. Let $x\in B$, $t>8r$ and $\alpha\in\R$. Then
\begin{equation}
\int_{8r\le|x-y|<t}\frac{|f(y)-f_B|}{|x-y|^{n-1}}\,dy
\le
\begin{cases}
ctr^\alpha\|f\|_{\mathcal {E}^{\alpha,1}(\Rn)}, &\alpha<0,\\
ct\log\frac{t}{r}\|f\|_{\mathcal {E}^{\alpha,1}(\Rn)}, &\alpha=0,\\
ct^{1+\alpha}\|f\|_{\mathcal {E}^{\alpha,1}(\Rn)}, &\alpha>0.
\end{cases}
\end{equation}
\end{lemma}
\begin{proof}
Let $k_0\in\N$ satisfy $2r\le 2^{-k_0}t<4r$.
Then, using Lemma \ref{lem:BMO-decreasing}, we have
\begin{align*}
&\int_{8r\le|x-y|<t}\frac{|f(y)-f_B|}{|x-y|^{n-1}}\,dy
\\
&=\sum_{k=0}^{k_0-1}
\int_{2^{-k-1}t\le|x-y|<2^{-k}t}\frac{|f(y)-f_B|}{|x-y|^{n-1}}\,dy
\\
&\le\sum_{k=0}^{k_0-1}\frac{1}{(2^{-k-1}t)^{n-1}}
\int_{B(x,2^{-k}t)}{|f(y)-f_B|}\,dy
\\
&\le c\sum_{k=0}^{k_0-1}
\frac{(2^{-k}t)^{n}}{(2^{-k-1}t)^{n-1}}\frac{1}{|B(x,2^{-k}t)|}
\int_{B(x,2^{-k}t)}{|f(y)-f_B|}\,dy
\\
&\le c\Bigl(\sum_{k=0}^{k_0-1}2^{-k}\Bigr)t\frac{1}{|B(x,2^{-k}t)|}
\int_{B(x,2^{-k}t)}{|f(y)-f_B|}\,dy
\\
&\le
\begin{cases}
ctr^\alpha\|f\|_{\mathcal {E}^{\alpha,1}(\Rn)}, &\alpha<0,\\
ct\log\frac{t}{r}\|f\|_{\mathcal {E}^{\alpha,1}(\Rn)}, &\alpha=0,\\
ct^{1+\alpha}\|f\|_{\mathcal {E}^{\alpha,1}(\Rn)}, &\alpha>0.
\end{cases}
\end{align*}
\end{proof}
\begin{lemma}\label{lem:BMO-n-1+beta}
Let $B=B(x_0,r)$. Let $x\in B$, $t>8r$, $0<\beta<1$ and $\alpha\in\R$.
\begin{equation}
\int_{8r\le|x-y|<t}\frac{r^\beta|f(y)-f_B|}{|x-y|^{n-1+\beta}}\,dy
\le
\begin{cases}
cr^{\beta+\alpha}t^{1-\beta}\|f\|_{\mathcal {E}^{\alpha,1}(\Rn)}, &\alpha<0,\\
cr^{\beta}t^{1-\beta}\log\frac{t}{r}\|f\|_{\mathcal {E}^{\alpha,1}(\Rn)},
&\alpha=0,\\
cr^\beta t^{1-\beta+\alpha}\|f\|_{\mathcal {E}^{\alpha,1}(\Rn)}, &\alpha>0.
\end{cases}
\end{equation}
If $\beta=1$, we have
\begin{equation}
\int_{8r\le|x-y|<t}\frac{r|f(y)-f_B|}{|x-y|^{n}}\,dy
\le
\begin{cases}
cr^{1+\alpha}\log\frac{t}{r}\|f\|_{\mathcal {E}^{\alpha,1}(\Rn)}, &\alpha<0,\\
cr\log^2\frac{t}{r}\|f\|_{\mathcal {E}^{\alpha,1}(\Rn)},
&\alpha=0,\\
cr t^{\alpha}\log\frac{t}{r}\|f\|_{\mathcal {E}^{\alpha,1}(\Rn)}, &\alpha>0.
\end{cases}
\end{equation}
\end{lemma}
\begin{proof}
Let $k_0\in\N$ satisfy $2r\le 2^{-k_0}t<4r$.
Then, using Lemma \ref{lem:BMO-decreasing}, we have
\begin{align*}
{}&\int_{8r\le|x-y|<t}\frac{r^\beta |f(y)-f_B|}{|x-y|^{n-1+\beta}}\,dy
\\
&=r^\beta\sum_{k=0}^{k_0-1}
\int_{2^{-k-1}t\le|x-y|<2^{-k}t}\frac{|f(y)-f_B|}{|x-y|^{n-1+\beta}}\,dy
\\
&\le r^\beta\sum_{k=0}^{k_0-1}\frac{1}{(2^{-k-1}t)^{n-1+\beta}}
\int_{B(x,2^{-k}t)}{|f(y)-f_B|}\,dy
\\
&\le cr^\beta\sum_{k=0}^{k_0-1}
\frac{(2^{-k}t)^{n}}{(2^{-k-1}t)^{n-1+\beta}}\frac{1}{|B(x,2^{-k}t)|}
\int_{B(x,2^{-k}t)}{|f(y)-f_B|}\,dy
\\
&\le cr^\beta \Bigl(\sum_{k=0}^{k_0-1}
2^{-k(1-\beta)}\Bigr)t^{1-\beta}\frac{1}{|B(x,2^{-k}t)|}
\int_{B(x,2^{-k}t)}{|f(y)-f_B|}\,dy
\\
&\le
\begin{cases}
cr^{\beta+\alpha}t^{1-\beta}\|f\|_{\mathcal {E}^{\alpha,1}(\Rn)}, &\alpha<0,
\\
cr^\beta t^{1-\beta}\log\frac{t}{r}\|f\|_{\mathcal {E}^{\alpha,1}(\Rn)},
&\alpha=0,\\
cr^\beta t^{1-\beta+\alpha}\|f\|_{\mathcal {E}^{\alpha,1}(\Rn)}, &\alpha>0.
\end{cases}
\end{align*}
\end{proof}

\begin{lemma}\label{lem:BMO-n+beta}
Let $m\in\N$. Let $B=B(x_0,r)$. Let $x\in B$, $0<\beta\le 1$, $\alpha\le 1$,
$\gamma<\beta$, and $\alpha+\gamma<\beta$ in the case $0<\alpha\le1$.
Suppose furthermore that
 $\alpha=\alpha_1+\cdots+\alpha_m$ and $\alpha_1,\cdots,\alpha_m$ satisfiy
one of the following three conditions:
{\rm(i)} $\alpha_1,\alpha_2,\dots,\alpha_m<0$ provided $\alpha<0$;
{\rm(ii)} $\alpha_1=\alpha_2=\dots=\alpha_m=0$ provided $\alpha=0$;
{\rm(iii)} $\alpha_1,\alpha_2,\dots,\alpha_m>0$ provided $\alpha>0$.
Then
\begin{equation}
\int_{((B(x,8r)^m)^c}\frac{r^\beta\prod_{i=1}^m|f_i(y_i)-(f_i)_B|}
{(\sum_{j=1}^{m}|x-y_j|)^{mn+\beta-\gamma}}\,d\vec{y}
\le
cr^{\alpha+\gamma}\prod_{i=1}^m\|f_i\|_{\mathcal {E}^{\alpha_i,1}(\Rn)}.
\end{equation}
\end{lemma}
\begin{proof}
Using Lemma \ref{lem:BMO-increasing}, we have
\begin{align*}
&\int_{((B(x,8r)^m)^c}\frac{r^\beta\prod_{i=1}^m|f_i(y_i)-(f_i)_B|}
{(\sum_{j=1}^{m}|x-y_j|)^{mn+\beta-\gamma}}\,d\vec{y}
\\
&= \sum_{k=0}^{\infty}
\int_{(B(x,2^{k+4}r))^m\setminus (B(x,2^{k+3}r))^m}
\frac{r^\beta\prod_{i=1}^m|f_i(y_i)-(f_i)_B|}
{(\sum_{j=1}^{m}|x-y_j|)^{mn+\beta-\gamma}}\,d\vec{y}
\\
&\le cr^\beta \sum_{k=0}^{\infty}\frac{1}{(2^{k+3}r)^{mn+\beta-\gamma}}
\int_{(B(x,2^{k+4}r))^m}\prod_{i=1}^m|f_i(y_i)-(f_i)_B|\,d\vec{y}
\\
&\le c r^\gamma\sum_{k=0}^{\infty}
\frac{1}{2^{k(\beta-\gamma)}}\prod_{i=1}^m\frac{1}{|B(x,2^{k+4}r)|}
\int_{B(x,2^{k+4}r)}|f_i(y_i)-(f_i)_B|\,dy_i
\\
&\le
\begin{cases}
c\sum_{k=0}^{\infty}
\frac{1}{2^{k(\beta-\gamma)}}r^{\alpha+\gamma}
\prod_{i=1}^m\|f_i\|_{\mathcal {E}^{\alpha_i,1}(\Rn)}, &\alpha<0,
\\
c\sum_{k=0}^{\infty}\frac{(k+4)^m}{2^{k(\beta-\gamma)}}r^\gamma
\prod_{i=1}^m\|f_i\|_{\mathcal {E}^{\alpha_i,1}(\Rn)}, &\alpha=0,
\\
c\sum_{k=0}^{\infty}\frac{2^{k\alpha}}{2^{k(\beta-\gamma)}}r^{\alpha+\gamma}
\prod_{i=1}^m\|f_i\|_{\mathcal {E}^{\alpha_i,1}(\Rn)}, &\alpha>0,
\end{cases}
\end{align*}
which implies the desired conclusion.
\end{proof}
\section{Proof of Theorem \ref{existence on BMO}}\label{proof 1.1}
\begin{proof}
It suffices to verify that for any $f_j\in BMO(\R^n)$, if there exists
$y_0\in\R^n$ such that $\mu(\vec{f})(y_0)<\infty$, then for any ball
$B\subset\R^n$, with $y_0\in B$,
$$
\frac{1}{|B|}\int_B |\mu(\vec{f})(x)-(\mu(\vec{f}))_B|\;dx
\leq C\prod_{j=1}^m\|f_j\|_{BMO}.
$$
For each fixed ball $B$ as above, let $r$ be its radius. Set
\begin{align}\label{<8r}
\mu^r(\vec{f})(x)=\left(\int_0^{8r} \left|\frac{1}{t^m}
\int_{(B(0,t))^m} \frac{\Omega(\vec{y})}{|\vec{y}|^{m(n-1)}}
\prod_{i=1}^m f_i(x-y_i)\; d\vec{y}\right|^2\;\frac{dt}{t}\right)^{1/2},
\end{align}
and
\begin{align}\label{>8r}
\mu^\infty(\vec{f})(x)=\left(\int_{8r}^\infty \left|\frac{1}{t^m}
\int_{(B(0,t))^m} \frac{\Omega(\vec{y})}{|\vec{y}|^{m(n-1)}}
\prod_{i=1}^m f_i(x-y_i)\; d\vec{y}\right|^2\;\frac{dt}{t}\right)^{1/2}.
\end{align}
By the vanishing condition (\ref{vanish}) of $\Omega$, we can see that for
any $x\in B$,
\begin{align*}
\mu^r(\vec{f})(x)=& \mu^r((f_1-(f_1)_B)\chi_{10B},\cdots,(f_m-(f_m)_B)
\chi_{10B})(x)
\\
\leq& \mu((f_1-(f_1)_B)\chi_{10B},\cdots,(f_m-(f_m)_B)\chi_{10B})(x).
\end{align*}
So, by the boundedness of $\mu$ in Theorem A, we have
$$
\int_B |\mu^r(\vec{f})(x)|^p \;dx
\leq C\prod_{j=1}^m \left(\int_{10B}|f_j(y_j)-(f_j)_B|^{p_j}\;dy_j
\right)^{\frac{p}{p_j}}\leq C|B|\prod_{j=1}^m\|f_j\|_{BMO}^p ,
$$
where $1<p_j<\infty,\,\frac{1}{p}=\frac{1}{p_1}+\cdots+\frac{1}{p_m}$. Then
$$
\frac{1}{|B|}\int_B |\mu^r(\vec{f})(x)| \;dx\leq C\prod_{j=1}^m\|f_j\|_{BMO}.
$$
Notice that,
\begin{align*}
\frac{1}{|B|}\int_B |\mu(\vec{f})(x)-(\mu(\vec{f}))_B|\;dx
\leq& 2\frac{1}{|B|}\int_B |\mu^r(\vec{f})(x)| \;dx
\\
&+2\frac{1}{|B|}\int_B |\mu^\infty(\vec{f})(x)-\inf_{y\in B}
\mu^\infty(\vec{f})(y)| \;dx.
\end{align*}
So we only need to show for any $x\in B$,
$$
|\mu^\infty(\vec{f})(x)-\inf_{y\in B}\mu^\infty(\vec{f})(y)|
\leq \sup_{y\in B}|\mu^\infty(\vec{f})(x)-\mu^\infty(\vec{f})(y)|
\leq C\prod_{j=1}^m\|f_j\|_{BMO}.
$$
Thus, to finish the proof of the theorem, it suffice to prove that for any
$x,z\in B$,
\begin{align}\label{destination}
|\mu^\infty(\vec{f})(x)-\mu^\infty(\vec{f})(z)|
\leq C\prod_{j=1}^m\|f_j\|_{BMO}.
\end{align}
It is easy to see that
\begin{align*}
|\mu^\infty(\vec{f})(x)-\mu^\infty(\vec{f})(z)|
&=\left|\left(\int_{8r}^\infty |F_t(\vec{f})(x)|^2\;\frac{dt}{t}
\right)^{\frac12}-\left(\int_{8r}^\infty |F_t(\vec{f})(z)|^2\;\frac{dt}{t}
\right)^{\frac12}\right|
\\
&\le \left(\int_{8r}^\infty |F_t(\vec{f})(x)
+F_t(\vec{f})(z)||F_t(\vec{f})(x)-F_t(\vec{f})(z)|\;\frac{dt}{t}
\right)^{\frac12}.
\end{align*}
Now the vanishing moment of $\Omega$ further tells us that for $z\in \R^n$
and $t_1,\dots,t_m>0$,
\begin{align}\label{estimate for F_t(f)}
&\left|\int_{\prod_{i=1}^m B(z,t_i)}
\frac{\Omega(z-y_1,\cdots,z-y_m)}{(\sum_{j=1}^m|z-y_j|)^{m(n-1)}}
\prod_{i=1}^m f_i(y_i)\; d\vec{y}\right|
\\
\leq& C\sum_{k=-\infty}^{-1} \frac{1}{\prod_{i=1}^m(2^{k+1}t_i)^{n-1}}
\int_{\prod_{i=1}^mB(z,2^{k+1}t_i)\setminus  \prod_{i=1}^mB(z,2^kt_i)}
\prod_{j=1}^m |f_j(y_j)-(f_j)_{B(z,2^{k+1}t)}|\;d\vec{y}\nonumber
\\
\leq& C\sum_{k=-\infty}^{-1} {\prod_{i=1}^m2^kt_i}
\prod_{j=1}^m\frac1{|B(z,2^{k+1}t_j)|}
\int_{B(z,2^{k+1}t_j)}
 |f_j(y_j)-(f_j)_{B(z,2^{k+1}t_i)}|\;dy_j\nonumber
\\
\leq& C \prod_{j=1}^mt_j \prod_{j=1}^m\|f_j\|_{BMO}.\nonumber
\end{align}
For $z\in\Rn$, $r>0$ and $t>8r$, $B(z,t)^m$ can be decomposed into the
following disjoint union
\begin{equation*}
B(z,t)^m=(B(z,t)\setminus B(z,8r))^m\cup (\cup_{i=1}^m B(z,t)^{i-1}\times
B(z,8r)\times B(z,t)^{m-i})\cup B(z,8r)^m.
\end{equation*}
So,
\begin{align}
&t^m F_t(\vec{f})(z)=\int_{B(z,t)^m}
\frac{\Omega(z-y_1,\cdots,z-y_m)}{(\sum_{j=1}^m|z-y_j|)^{m(n-1)}}
\prod_{i=1}^m f_i(y_i)\; d\vec{y}\label{eq:F_t-decomp}
\\
&=\int_{(B(z,t)\setminus B(z,8r))^m}
\frac{\Omega(z-y_1,\cdots,z-y_m)}{(\sum_{j=1}^m|z-y_j|)^{m(n-1)}}
\prod_{i=1}^m f_i(y_i)\; d\vec{y}\notag
\\
&+\sum_{\ell=1}^{m}\int_{B(z,t)^{\ell-1}\times B(z,8r)\times B(z,t)^{m-\ell}}
\frac{\Omega(z-y_1,\cdots,z-y_m)}{(\sum_{j=1}^m|z-y_j|)^{m(n-1)}}
\prod_{i=1}^m f_i(y_i)\; d\vec{y}\notag
\\
&-(m-1)\int_{B(z,8r)^m}
\frac{\Omega(z-y_1,\cdots,z-y_m)}{(\sum_{j=1}^m|z-y_j|)^{m(n-1)}}
\prod_{i=1}^m f_i(y_i)\; d\vec{y}.\notag
\end{align}
By \eqref{estimate for F_t(f)} we see that
\begin{align}
&t^m |F_t(\vec{f})(z)|\le Ct^m\prod_{j=1}^m\|f_j\|_{BMO},\label{eq:Ft-BMO}
\\
&\biggl|\int_{B(z,t)^{\ell-1}\times B(z,8r)\times B(z,t)^{m-\ell}}
\frac{\Omega(z-y_1,\cdots,z-y_m)}{(\sum_{j=1}^m|z-y_j|)^{m(n-1)}}
\prod_{i=1}^m f_i(y_i)\; d\vec{y}\biggr|\label{eq:B8r-i-BMO}
\\
&\hspace{6cm}\le Ct^{m-1}r\prod_{j=1}^m\|f_j\|_{BMO},\quad\ell=1,\dots,m,
\notag
\\
&\left|\int_{B(z,8r)^m}
\frac{\Omega(z-y_1,\cdots,z-y_m)}{(\sum_{j=1}^m|z-y_j|)^{m(n-1)}}
\prod_{i=1}^m f_i(y_i)\; d\vec{y}\right|
\le Cr^m \prod_{j=1}^m\|f_j\|_{BMO}.\label{eq:B8r-BMO}
\end{align}

For any fixed $x\in B$ and $t\geq 8r$, set
\begin{align}
&H_t(\vec{f})(x,z):=\biggl|
\int_{(B(z,t)\setminus B(z,8r))^m}
\frac{\Omega(z-y_1,\cdots,z-y_m)}{(\sum_{j=1}^m|z-y_j|)^{m(n-1)}}
\prod_{i=1}^m f_i(y_i)\; d\vec{y}\label{eq:Ht-f}
\\
&\hspace{4cm}-
\int_{(B(x,t)\setminus B(x,8r))^m}
\frac{\Omega(x-y_1,\cdots,x-y_m)}{(\sum_{j=1}^m|x-y_j|)^{m(n-1)}}
\prod_{i=1}^m f_i(y_i)\; d\vec{y}\biggr|.\notag
\end{align}
Then by \eqref{eq:F_t-decomp}, \eqref{eq:Ft-BMO}, \eqref{eq:B8r-i-BMO},
\eqref{eq:B8r-BMO} and \eqref{eq:Ht-f} we have for any $x,z\in B$,
\begin{align}
&|\mu^\infty(\vec{f})(x)-\mu^\infty(\vec{f})(z)|\label{eq:bound-by-Ht}
\\
&\leq C\left(
\int_{8r}^\infty |F_t(\vec{f})(x)-F_t(\vec{f})(z)|\;\frac{dt}{t}\right)^{1/2}
\prod_{j=1}^m\|f_j\|_{BMO}^{\frac{1}{2}}\notag
\\
&\leq C\prod_{j=1}^m\|f_j\|_{BMO}^{\frac{1}{2}}\left(
\int_{8r}^\infty (t^{m-1}r+r^m)\;\frac{dt}{t^{m+1}}\right)^{1/2}\notag
\\
&\hspace{3cm}
+C\prod_{j=1}^m\|f_j\|_{BMO}^{\frac{1}{2}}\left(
\int_{8r}^\infty |H_t(\vec{f})(x,z)|\;\frac{dt}{t^{m+1}}\right)^{1/2}\notag
\\
&\leq C\big[\prod_{j=1}^m\|f_j\|_{BMO}
+\prod_{j=1}^m\|f_j\|_{BMO}^{\frac{1}{2}}\left(
\int_{8r}^\infty |H_t(\vec{f})(x,z)|\;\frac{dt}{t^{m+1}}\right)^{1/2}.
\notag
\end{align}
Therefore, the proof of inequality (\ref{destination}) can be reduced to
proving that for any $x,z\in B$,
\begin{equation}\label{eq:purpose-1}
\int_{8r}^\infty |H_t(\vec{f})(x,z)|\;\frac{dt}{t^{m+1}}
\leq C\prod_{j=1}^m\|f_j\|_{BMO}.
\end{equation}
To show this we first introduce some notations. We fix $x$ and $z$,
and for $t>0$ write
$$\Xi(x,t)=\{y\in \R^n: 8r\leq |x-y|<t,\,8r\leq |z-y|<t\};$$
$$\Xi(z,t)=\{y\in \R^n: 8r\leq |z-y|<t,\,8r\leq |x-y|<t\};$$
$$\Gamma(x,t)=\{y\in \R^n: 8r\leq |x-y|<t,\,|z-y|\geq t\};$$
$$\Gamma(z,t)=\{y\in \R^n: 8r\leq |z-y|<t,\,|x-y|\geq t\};$$
$$\Lambda(x,t)=\{y\in \R^n: 8r\leq |x-y|<t,\,|z-y|< 8r\};$$
$$\Lambda(z,t)=\{y\in \R^n: 8r\leq |z-y|<t,\,|x-y|< 8r\};$$
$$\vec{\Theta}(x,t)=\Theta_1(x,t)\times\cdots\times \Theta_m(x,t),\,
\Theta_i(x,t)\in\{\Xi(x,t),\,\Gamma(x,t),\,\Lambda(x,t)\};$$
$$\vec{\Theta}(z,t)=\Theta_1(z,t)\times\cdots\times \Theta_m(z,t),\,
\Theta_i(z,t)\in\{\Xi(z,t),\,\Gamma(z,t),\,\Lambda(z,t)\}.$$

For any $y$, denote
$$\Xi(x,y)=\{t>0: 8r\leq |x-y|<t,\, 8r\leq |z-y|<t\};$$
$$\Xi(z,y)=\{t>0: 8r\leq |z-y|<t,\, 8r\leq |x-y|<t\};$$
$$\Gamma(x,y)=\{t>0: 8r\leq |x-y|<t,\,|z-y|\geq t\};$$
$$\Gamma(z,y)=\{t>0: 8r\leq |z-y|<t,\,|x-y|\geq t\};$$
$$\Lambda(x,y)=\{t>0: 8r\leq |x-y|<t,\,|z-y|< 8r\};$$
$$\Lambda(z,y)=\{t>0: 8r\leq |z-y|<t,\,|x-y|< 8r\};$$

$$\Lambda_i(x,y_i)\in\{\Gamma(x,y_i),\,\Xi(x,y_i)\},\,i=1,\cdots,m;$$
$$\Lambda_i(z,y_i)\in\{\Gamma(z,y_i),\,\Xi(z,y_i)\},\,i=1,\cdots,m.$$

Moreover, some immediate consequences are
$$B(x,t)\backslash B(x,8r)=\Xi(x,t)\cup\Gamma(x,t)\cup\Lambda(x,t),$$
$$B(z,t)\backslash B(z,8r)=\Xi(z,t)\cup\Gamma(z,t)\cup\Lambda(z,t)$$
and $$\Xi(x,t)=\Xi(z,t)=:\Xi(t),\,\Xi(x,y)=\Xi(z,y)=:\Xi(y).$$
Using these notations, we have
\begin{align*}
{}&H_t(\vec{f})(x,z)
\\ 
&=\biggl|\int_{(B(x,t)\backslash B(x,8r))^m}
\frac{\Omega(x-y_1,\cdots,x-y_m)}{(\sum_{j=1}^m|x-y_j|)^{m(n-1)}}
\prod_{i=1}^m f_i(y_i)\; d\vec{y}
\\ 
&\quad-\int_{(B(z,t)\backslash B(z,8r))^m}
\frac{\Omega(z-y_1,\cdots,z-y_m)}{(\sum_{j=1}^m|z-y_j|)^{m(n-1)}}
\prod_{i=1}^m f_i(y_i)\; d\vec{y}\biggr|
\\ 
&
\le \int_{(\Xi(t))^m}\left|
\frac{\Omega(x-y_1,\cdots,x-y_m)}{(\sum_{j=1}^m|x-y_j|)^{m(n-1)}}
-\frac{\Omega(z-y_1,\cdots,z-y_m)}{(\sum_{j=1}^m|z-y_j|)^{m(n-1)}}\right|
\\ 
&\hspace{8cm}\times\prod_{i=1}^m |f_i(y_i)-(f_i)_B|\; d\vec{y}
\\ 
&\quad+\int_{(\Lambda(x,t))^m}
\frac{|\Omega(x-y_1,\cdots,x-y_m)|}{(\sum_{j=1}^m|x-y_j|)^{m(n-1)}}
\prod_{i=1}^m |f_i(y_i)-(f_i)_B|\; d\vec{y}
\\ 
&\quad+\int_{(\Lambda(z,t))^m}
\frac{|\Omega(z-y_1,\cdots,z-y_m)|}{(\sum_{j=1}^m|z-y_j|)^{m(n-1)}}
\prod_{i=1}^m |f_i(y_i)-(f_i)_B|\; d\vec{y}
\\ 
&\quad+\int_{\substack{
\vec{\Theta}(x,t)
\\ 
\exists \Theta_i(x,t)=\Gamma(x,t)}}
 \frac{|\Omega(x-y_1,\cdots,x-y_m)|}{(\sum_{j=1}^m|x-y_j|)^{m(n-1)}}
\prod_{i=1}^m |f_i(y_i)-(f_i)_B|\; d\vec{y}
\\ 
&\quad+\int_{\substack{
\vec{\Theta}(z,t)\\
\exists \Theta_i(z,t)=\Gamma(z,t)}}
\frac{|\Omega(z-y_1,\cdots,z-y_m)|}{(\sum_{j=1}^m|z-y_j|)^{m(n-1)}}
\prod_{i=1}^m |f_i(y_i)-(f_i)_B|\; d\vec{y}
\\ 
&\quad+\sum_{l=1}^{m-1} \int_{(\Xi(x,t))^l}
\int_{(\Lambda(x,t))^{m-l}}
\frac{|\Omega(x-y_1,\cdots,x-y_m)|}{(\sum_{j=1}^m|x-y_j|)^{m(n-1)}}
\prod_{i=1}^m |f_i(y_i)-(f_i)_B|\; d\vec{y}
\\
&\quad+\sum_{l=1}^{m-1} \int_{(\Xi(z,t))^l} \int_{(\Lambda(z,t))^{m-l}}
\frac{|\Omega(z-y_1,\cdots,z-y_m)|}{(\sum_{j=1}^m|z-y_j|)^{m(n-1)}}
\prod_{i=1}^m |f_i(y_i)-(f_i)_B|\; d\vec{y}
\\ 
&
=:\sum_{i=1}^5 H_{t,i}(\vec{f})(x,z)
+\sum_{l=1}^{m-1} H_{t,6}^l(\vec{f})(x,z)
+\sum_{l=1}^{m-1} H_{t,7}^l(\vec{f})(x,z) .
\end{align*}
In the above, we did not explicitly write all the permutated terms
for the sake of simplicity.

For $x,z\in B$, we have by Lemma \ref{lem:BMO-n-1}
\begin{align*}
|H_{t,2}(\vec{f})(x,z)|\leq &C\int_{(B(x,10r)\backslash B(x,8r))^m}
\frac{\prod_{i=1}^m |f_i(y_i)-(f_i)_B|}{(\sum_{j=1}^m|x-y_j|)^{m(n-1)}}
\;d\vec{y}
\\
\leq& C \prod_{i=1}^m \int_{8r\le|x-y_i|\le10r}
\frac{|f_i(y_i)-(f_i)_B|}{|x-y_i|^{n-1}} \;dy_i
\\
\leq& Cr^m \prod_{j=1}^m\|f_j\|_{BMO},
\end{align*}
which leads to
$$
\int_{8r}^\infty |H_{t,2}(\vec{f})(x,z)|\;\frac{dt}{t^{m+1}}
\leq C \prod_{j=1}^m\|f_j\|_{BMO},
$$
and similarly,
$$
\int_{8r}^\infty |H_{t,3}(\vec{f})(x,z)|\;\frac{dt}{t^{m+1}}
\leq C \prod_{j=1}^m\|f_j\|_{BMO}.
$$
For $H_{t,4}(\vec{f})$, note that for any $x,z\in B$ the length of
$t$ can be controlled by
$$
|\cap_{j=1}^m \Theta_j(x,y_j)|\leq |\Theta_i(x,y_i)|=|\Gamma(x,y_i)|
\leq ||z-y_i|-|x-y_i||\leq |z-x|\leq 2r,
$$
and for any $t\in \cap_{j=1}^m \Theta_j(x,y_j)$,
$t>\frac{1}{m}(\sum_{j=1}^m|x-y_j|)$.
Then we can obtain the following estimate by using Lemma \ref{lem:BMO-n+beta}
\begin{align}\label{previousH_{t,4}}
&\int_{8r}^\infty |H_{t,4}(\vec{f})(x,z)|\;\frac{dt}{t^{m+1}}
\\
&\leq C\int_{((B(x,8r))^c)^m}
\frac{\prod_{i=1}^m |f_i(y_i)-(f_i)_B|}{(\sum_{j=1}^m|x-y_j|)^{m(n-1)}}
\int_{\bigcap_{j=1}^m \Theta_j(x,y_j)} \frac{dt}{t^{m+1}}\;d\vec{y}\notag
\\
&
\leq C r \int_{((B(x,8r))^c)^m}
\frac{\prod_{i=1}^m |f_i(y_i)-(f_i)_B|}{(\sum_{j=1}^m|x-y_j|)^{mn+1}}
\;d\vec{y}\nonumber
\\
&\le C \prod_{j=1}^m\|f_j\|_{BMO},\nonumber
\end{align}
For $H_{t,5}(\vec{f})$, similarly, we have
\begin{align}\label{H_{t,5}}
\int_{8r}^\infty |H_{t,5}(\vec{f})(x,z)|\;\frac{dt}{t^{m+1}}
\leq C \prod_{j=1}^m\|f_j\|_{BMO},
\end{align}
For $H_{t,6}^l(\vec{f})$, we may assume $y_1,\cdots,y_l\in \Xi(t)$ and
$y_{l+1},\cdots,y_m\in \Lambda(x,t)$. Then by the simple calculation and by
Lemma \ref{lem:BMO-n-1}, we have
\begin{align}
&H_{t,6}^l(\vec{f})(x,z)\label{H_{t,6}^l-0}
\\
=&\int_{(\Xi(x,t))^l} \int_{(\Lambda(x,t))^{m-l}}
\frac{|\Omega(x-y_1,\cdots,x-y_m)|}{(\sum_{j=1}^m|x-y_j|)^{m(n-1)}}
\prod_{i=1}^m |f_i(y_i)-(f_i)_B|\; d\vec{y} \notag
\\
\leq &C \prod_{j=1}^l
 \int_{8r\leq |x-y_j|<t}
\frac{|f_j(y_j)-(f_j)_B|}{|x-y_j|^{n-1}}\;dy_j
\prod_{j=l+1}^m \int_{8r\leq |x-y_j|
 \leq 10r} \frac{|f_j(y_j)-(f_j)_B|}{|x-y_j|^{n-1}}\;dy_j \notag
\\
\leq& C\Bigl(t\log_2\frac{t}{r}\Bigr)^l r^{m-l} \prod_{j=1}^m\|f_j\|_{BMO}.
 \notag
\end{align}
So, we obtain
\begin{align}\label{H_{t,6}^l}
\int_{8r}^\infty |H_{t,6}^l(\vec{f})(x,z)|\;\frac{dt}{t^{m+1}}
&\leq C  \int_{8r}^\infty r^{m-l}\Bigl(t\log_2\frac{t}{r}\Bigr)^l
\;\frac{dt}{t^{m+1}}
\\
&\le C  \int_{8}^\infty\frac{(\log_2s)^l}{s^{m-l+1}}\,ds
\leq C\prod_{j=1}^m\|f_j\|_{BMO}.\notag
\end{align}
Similar estimate holds for $H_{t,7}^l(\vec{f})$.

It remains to estimate $H_{t,1}(\vec{f})$. We employ the Lipschitz continuous
condition (ii) of $\Omega$ and Lemma \ref{lem:BMO-n+beta}
to get the following.
\begin{align}\label{H_{t,1}}
&\int_{8r}^\infty |H_{t,1}(\vec{f})(x,z)|\;\frac{dt}{t^{m+1}}
\\
\leq& C\int_{((B(x,8r))^c)^m}
\frac{|x-z|^\alpha}{(\sum_{j=1}^m|x-y_j|)^{m(n-1)+\alpha}}
\prod_{i=1}^m |f_i(y_i)-(f_i)_B|
\int_{\frac{1}{m}(\sum_{j=1}^m|x-y_j|)}^\infty \frac{dt}{t^{m+1}}\;d\vec{y}
\nonumber
\\
\leq& C\int_{((B(x,8r))^c)^m}
\frac{r^\alpha}{(\sum_{j=1}^m|x-y_j|)^{mn+\alpha}}
\prod_{i=1}^m |f_i(y_i)-(f_i)_B|\;d\vec{y}.\nonumber
\\
\le& C\prod_{j=1}^m\|f_j\|_{BMO}.\notag
\end{align}
Thus, we have proved (\ref{destination}). This completes the proof of
Theorem \ref{existence on BMO}.

\end{proof}
\section{Proofs of Theorems \ref{existence on Campanato space}-\ref{existence on Lipschitz space}}\label{Sec-4}

\begin{proof}[Proof of Theorem 1.2]
Similarly to the proof of Theorem \ref{existence on BMO}, to prove
Theorem \ref{existence on Campanato space}, it suffices to show that for any
$f_j\in\mathcal{E}^{\alpha_j,p_j}(\R^n)$
with $\|f_j\|_{\mathcal{E}^{\alpha_j,p_j}}=1$, if there exists $y_0\in\R^n$
such that $\mu(\vec{f})(y_0))<\infty$, then for any ball $B\subset\R^n$
with $B\ni y_0$,
$$
\left(\frac{1}{|B|}\int_B |\mu(\vec{f})(x)-\inf_{y\in B}\mu(\vec{f})(y)|^p\;dx
\right)^{1/p} \leq C|B|^{\alpha/n}.
$$
Let $r$ be the radius of $B$, $\mu^r(\vec{f})$ and $\mu^\infty(\vec{f})$ be
the same as in (\ref{<8r}) and (\ref{>8r}), respectively. Since,
\begin{align*}
|\mu(\vec{f})(x)-\inf_{y\in B}\mu(\vec{f})(y)|
&\le \mu(\vec{f})(x)-\inf_{y\in B}\mu^\infty(\vec{f})(y)
\\
&\leq |\mu^r(\vec{f})(x)|+\sup_{y\in B}|\mu^\infty(\vec{f})(x)-\mu^\infty
(\vec{f})(y)|,
\end{align*}

by the vanishing moment of $\Omega$, we can write
\begin{align*}
&\left(\frac{1}{|B|}\int_B |\mu(\vec{f})(x)-\inf_{y\in B}\mu(\vec{f})(y)|^p\;dx
\right)^{1/p}
\\
\leq& \left(\frac{1}{|B|}\int_B |\mu^r(\vec{f})(x)|^p\;dx\right)^{1/p}
+\left(\frac{1}{|B|}\int_B \sup_{y\in B}|\mu^\infty(\vec{f})(x)-\mu^\infty
(\vec{f})(y)|^p\;dx\right)^{1/p}
\\
\leq& \left(\frac{1}{|B|}\int_B |\mu^r((f_1-(f_1)_{10B})
\chi_{10B},\cdots,(f_m-(f_m)_{10B})\chi_{10B})(x)|^p\;dx\right)^{1/p}
\\
&+\left(\frac{1}{|B|}\int_B \sup_{y\in B}
|\mu^\infty(\vec{f})(x)-\mu^\infty(\vec{f})(y)|^p\;dx\right)^{1/p}
\\
=&:I+II.
\end{align*}
By Theorem A, we can write
$$
I\leq C \frac{1}{|B|^{1/p}} \left(\prod_{j=1}^m
\int_{10B}|f_j(y_j)-(f_j)_{10B}|^{p_j}\right)^{1/p_j}\leq C|B|^{\alpha/n}.
$$
Thus, the proof of Theorem \ref{existence on Campanato space} is now reduced
to prove that for any $x,z\in B$,
\begin{align}\label{destination2}
|\mu^\infty(\vec{f})(x)-\mu^\infty(\vec{f})(z)|\leq C r^\alpha.
\end{align}
If $-m<\alpha<\infty$, $p\in(1,\infty)$ a standard computation gives us that
for any $z\in B$ and $t_1,\dots, t_m>0$,
\begin{align}\label{estimate for F_t(f)2}
&\left|\int_{\prod_{i=1}^m(B(z,t_i)}
\frac{\Omega(z-y_1,\cdots,z-y_m)}{(\sum_{j=1}^m|z-y_j|)^{m(n-1)}}
\prod_{i=1}^m f_i(y_i)\; d\vec{y}\right|
\\
\leq& C\sum_{k=-\infty}^{0} \frac{1}{\prod_{i=1}^m(2^kt_i)^{n-1}}
\int_{\prod_{i=1}^mB(z,2^kt_i)\setminus  \prod_{i=1}^mB(z,2^{k-1}t_i)}
\prod_{j=1}^m |f_j(y_j)-(f_j)_{B(z,2^kt_j)}|\;d\vec{y}\nonumber
\\
\leq& C\sum_{k=-\infty}^{0} \prod_{i=1}^m(2^kt_i)
\left(\frac{1}{\prod_{i=1}^m|B(z,2^kt_i)|}\int_{\prod_{i=1}^mB(z,2^kt_i)}
\prod_{j=1}^m |f_j(y_j)-(f_j)_{B(z,2^kt_j)}|^p\;d\vec{y}
\right)^{1/p}\nonumber
\\
\leq& C\prod_{j=1}^mt_j \sum_{k=-\infty}^{0} 2^{km}
\prod_{j=1}^m \left(\frac{1}{|B(z,2^kt_j)|^m}\int_{(B(z,2^kt))^m}
|f_j(y_j)-(f_j)_{B(z,2^kt_j)}|^{p_j}\;d\vec{y}\right)^{1/p_j}\nonumber
\\
\leq& C\prod_{j=1}^mt_j \sum_{k=-\infty}^{0} 2^{km}
\prod_{j=1}^m \left(\frac{1}{|B(z,2^kt_j)|}\int_{(B(z,2^kt_j))}
|f_j(y_j)-(f_j)_{B(z,2^kt_j)}|^{p_j}\;dy_j\right)^{1/p_j}\nonumber
\\
\leq& C \prod_{j=1}^mt_j^{1+\alpha_j} .\nonumber
\end{align}
(We shall use this fact to prove Theorem \ref{existence on Lipschitz space}
for $0<\alpha\le1$.)

On the other hand, if $p\in(n,\infty)$ and $\alpha\in(-\infty,0)$,
it follows from H\"{o}lder's inequality that
\begin{align}\label{estimate for F_t(f)3}
&\left|\int_{\prod_{i=1}^m B(z,t_i)}
\frac{\Omega(z-y_1,\cdots,z-y_m)}{(\sum_{j=1}^m|z-y_j|)^{m(n-1)}}
\prod_{i=1}^m f_i(y_i)\; d\vec{y}\right|
\\
\leq& C\left(\int_{\prod_{i=1}^m B(z,t_i)}
\prod_{j=1}^m |f_j(y_j)-(f_j)_{B(z,t_j)}|^p\;d\vec{y}\right)^{1/p} \notag
\\
&\hspace{3cm}\times
\left(\int_{\prod_{i=1}^m B(z,t_i)}
\frac{d\vec{y}}{(\sum_{j=1}^m|z-y_j|)^{m(n-1)p^\prime}}
\right)^{1/p^\prime}\nonumber
\\
\leq& C\prod_{j=1}^mt_j^{\frac{n}{p^\prime}-n+1}
\prod_{j=1}^m\left(\int_{\prod_{i=1}^m B(z,t_i)}
|f_j(y_j)-(f_j)_{B(z,t_j)}|^{p_j}\;d\vec{y}
\right)^{1/p_j}\nonumber
\\
\leq& C\prod_{j=1}^mt_j^{\frac{n}{p^\prime}-n+1}
\prod_{j=1}^m (\prod_{i\ne j}t_i^{n/p_j} t_j^{\alpha_j+n/p_j})
=C \prod_{j=1}^mt_j^{\frac{n}{p^\prime}-n+1}
\prod_{j=1}^m t_j^{n/p} t_j^{\alpha_j+n/p_j})\nonumber
\\
=&C\prod_{j=1}^mt_j^{1+\alpha_j}.\nonumber
\end{align}
Hence for any $x,z\in B$, according to (\ref{estimate for F_t(f)2}) and
(\ref{estimate for F_t(f)3}),
 we see that
\begin{align}
&t^m |F_t(\vec{f})(z)|\le Ct^{m+\alpha}
\prod_{j=1}^m\|f_j\|_{\mathcal{E}^{\alpha_j,p_j}},\label{eq:Ft-Campanato}
\\
&\biggl|\int_{B(z,t)^{\ell-1}\times B(z,8r)\times B(z,t)^{m-\ell}}
\frac{\Omega(z-y_1,\cdots,z-y_m)}{(\sum_{j=1}^m|z-y_j|)^{m(n-1)}}
\prod_{i=1}^m f_i(y_i)\; d\vec{y}\biggr|\label{eq:B8r-i-Campanato}
\\
&\hspace{4cm}\le Ct^{m-1+\alpha-\alpha_\ell}r^{1+\alpha_\ell}
\prod_{j=1}^m\|f_j\|_{\mathcal{E}^{\alpha_j,p_j}},
\quad\ell=1,\dots,m,\notag
\\
&\left|\int_{B(z,8r)^m}
\frac{\Omega(z-y_1,\cdots,z-y_m)}{(\sum_{j=1}^m|z-y_j|)^{m(n-1)}}
\prod_{i=1}^m f_i(y_i)\; d\vec{y}\right|
\le Cr^{m+\alpha} \prod_{j=1}^m\|f_j\|_{\mathcal{E}^{\alpha_j,p_j}}.
\label{eq:B8r-Campanato}
\end{align}
Therefore by \eqref{eq:F_t-decomp}, \eqref{eq:Ft-Campanato},
\eqref{eq:B8r-i-Campanato}, \eqref{eq:B8r-Campanato} and \eqref{eq:Ht-f}
we have for any $x,z\in B$,
\begin{align*}
{}&|\mu^\infty(\vec{f})(x)-\mu^\infty(\vec{f})(z)|
\\
&\le C\left(\int_{8r}^\infty |F_t(\vec{f})(x)-F_t(\vec{f})(z)|\;
\frac{dt}{t^{1-\alpha}}\right)^{\frac12}
\\
&\le C\left(\int_{8r}^\infty \sum_{j=1}^{m}t^{m-1+\alpha-\alpha_j}
r^{1+\alpha_j}+r^{m+\alpha}\;\frac{dt}{t^{m+1-\alpha}}\right)^{\frac12}
+\left(\int_{8r}^\infty |H_t(\vec{f})(x,z)|\;\frac{dt}{t^{m+1-\alpha}}
\right)^{\frac12}
\\
&\le C r^\alpha+\left(\int_{8r}^\infty |H_t(\vec{f})(x,z)|\;
\frac{dt}{t^{m+1-\alpha}}\right)^{\frac12},
\end{align*}
where $H_t(\vec{f})(x,z)$ is the same as in the proof of Theorem
\ref{existence on BMO}. Again decompose $H_t(\vec{f})(x,z)$ into
$$
H_t(\vec{f})(x,z)\leq \sum_{i=1}^5 H_{t,i}(\vec{f})(x,z)
+\sum_{l=1}^{m-1} H_{t,6}^l(\vec{f})(x,z)
+\sum_{l=1}^{m-1} H_{t,7}^l(\vec{f})(x,z).
$$

Applying H\"{o}lder's inequality and the Lipschitz continuous condition of
$\Omega$, we obtain by Lemma \ref{lem:BMO-n+beta} that for any $x,z\in B$,
\begin{align*}
&\int_{8r}^\infty |H_{t,1}(\vec{f})(x,z)|\;\frac{dt}{t^{m+1-\alpha}}
\\
&\leq C\int_{((B(x,8r))^c)^m}
\frac{|x-z|^\beta}{(\sum_{j=1}^m|x-y_j|)^{m(n-1)+\beta}}
\prod_{i=1}^m |f_i(y_i)-(f_i)_B|
\int_{\frac{1}{m}(\sum_{j=1}^m|x-y_j|)}^\infty \frac{dt}{t^{m+1-\alpha}}\;
d\vec{y}\nonumber
\\
&\leq C\int_{((B(x,8r))^c)^m}
\frac{r^{\beta}|f_i(y_i)-(f_i)_B|}{(\sum_{j=1}^m|x-y_j|)^{mn+\beta-\alpha}}
 \;d\vec{y}
\\
&\le Cr^{2\alpha}.
\end{align*}
On the other hand, by Lemma \ref{lem:BMO-n-1} we get
\begin{equation*}
|H_{t,2}(\vec{f})(x,z)|\le C\prod_{i=1}^m \int_{8r\le|x-y_i|\le10r}
\frac{|f_i(y_i)-(f_i)_B|}{|x-y_i|^{n-1}} \;d{yi}
\le Cr^{m+\alpha}.
\end{equation*}
Hence
\begin{align*}
{}\int_{8r}^\infty |H_{t,2}(\vec{f})(x,z)|\;\frac{dt}{t^{m+1-\alpha}}
\leq C\int_{8r}^\infty \frac{dt}{t^{m+1-\alpha}}  r^{m+\alpha}
\leq C r^{2\alpha}.
\end{align*}
Similarly we get
\begin{equation*}
\int_{8r}^\infty |H_{t,3}(\vec{f})(x,z)|\;\frac{dt}{t^{m+1-\alpha}}
\le C r^{2\alpha}.
\end{equation*}
Like as in getting \eqref{previousH_{t,4}} in the proof of Theorem
\ref{existence on BMO}. we obtain by using Lemma \ref{lem:BMO-n+beta}
\begin{align*}
{}\int_{8r}^\infty |H_{t,4}(\vec{f})(x,z)|\;\frac{dt}{t^{m+1-\alpha}}
&\le Cr\int_{((B(x,8r))^c)^m}
\frac{\prod_{i=1}^m |f_i(y_i)-(f_i)_B|}{(\sum_{j=1}^m|x-y_j|)^{mn+1-\alpha}}\;
d\vec{y}
\\
&\le Cr^{2\alpha}.
\end{align*}
Similar estimate holds for $H_{t,5}(\vec{f})(x,z)$.

As for $H_{t,6}^l(\vec{f})(x,z)$ and $H_{t,7}^l(\vec{f})(x,z)$,
 similarly by using Lemma \ref{lem:BMO-n-1} we have
\begin{align*}
&\int_{8r}^\infty |H_{t,6}^l(\vec{f})(x,z)|\;\frac{dt}{t^{m+1-\alpha}}
+ \int_{8r}^\infty |H_{t,7}^l(\vec{f})(x,z)|\;\frac{dt}{t^{m+1-\alpha}}
\\
&\leq C r^{m-l+\alpha} \int_{8r}^\infty t^{l}\;\frac{dt}{t^{m+1-\alpha}}
\leq Cr^{2\alpha}
\end{align*}
Combine the estimates for
$H_{t,i}(\vec{f})(x,z) (1\leq i\leq 5),\,H_{t,6}^l(\vec{f})$ and
$H_{t,7}^l(\vec{f})$, we obtain that for any $x,z\in B$, (\ref{destination2})
holds. We complete the proof of Theorem \ref{existence on Campanato space}.
\end{proof}

\begin{proof}[Proof of Theorem 1.3] 
Using \eqref{estimate for F_t(f)2}, we can prove
Theorem \ref{existence on Lipschitz space} similarly to the proof of
Theorem \ref{existence on Campanato space}. The details are omitted.
\end{proof}

\section{Extension to operators with separated kernels}\label{Sec-5}

\begin{definition}\label{multilinear Marcinkiewicz integral with separated
kernels} (multilinear Marcinkiewicz integral with separated kernels).\\
Let $\Omega=\prod_{j=1}^m\Omega_j$ be a function defined on $(\R^n)^m$
with the following properties:

(i) For $j=1,\cdots,m$, $\Omega_j$ is homogeneous of degree $0$ on $\Rn$,
i.e. for any $\lambda>0$ and $y\in \R^n$,
\begin{align}\label{homogeneous1}
\Omega_j(\lambda y)=\Omega(y);
\end{align}

(ii) $\Omega_j$ is Lipschitz continuous on $S^{n-1}$, i.e. there is
$0<\alpha<1$ and $C>0$ such that for any $\xi,\,\eta\in \R^n$
\begin{align}\label{Lipschitz}
|\Omega_j(\xi)-\Omega_j(\eta)|\leq C|\xi^\prime-\eta^\prime|^\alpha,\quad
 j=1,\cdots,m.
\end{align}
where $y^\prime=\frac{y}{|y|}$;

(iii) The integration of $\Omega_j$ on each unit sphere vanishes,
\begin{align}\label{vanishing1}
\int_{S^{n-1} }\Omega_j(y) \; dy=0,\quad j=1,\cdots,m.
\end{align}
For any $\vec{f}=(f_1,\cdots,f_m)\in S\times \cdots \times S$,
we define the operator $F_t$ for any $t>0$ as
\begin{align}\label{F_t-1}
F_t(\vec{f})(x)
&=\frac{\chi_{(B(0,t))^m} \Omega({\cdot})}{t^m\prod_{j=1}^m|{\cdot_j}|^{n-1}}
\ast(f_1\otimes \cdots \otimes f_m)(x)
\\
&=\frac{1}{t^m}\int_{(B(0,t))^m}
\prod_{j=1}^m\frac{\Omega_j({y_j})}{|{y_j}|^{n-1}} \prod_{j=1}^m f_j(x-y_j)\;
d\vec{y},\nonumber
\\
&=\frac{1}{t^m}\prod_{j=1}^m \int_{B(0,t)}
\frac{\Omega_j({y_j})}{|{y_j}|^{n-1}} f_j(x-y_j)\; dy_j,\nonumber
\end{align}
Finally,
the multilinear Marcinkiewicz integral $\tilde{\mu}$ is defined by
\begin{align}\label{Marcinkiewicz integral1}
\tilde{\mu}(\vec{f})(x)=\left(\int_0^\infty |F_t(\vec{f})(x)|^2\;\frac{dt}{t}
\right)^{1/2}.
\end{align}
\end{definition}
It is easily seen that
\begin{equation*}
\tilde{\mu}(\vec{f})(x)\le C \biggl(\int_0^\infty \Bigl|
\frac{1}{t}\int_{B(0,t)}
\frac{\Omega_1({y})}{|{y}|^{n-1}} f_1(x-y_1)\Bigr|^2\;\frac{dt}{t}
\biggr)^{1/2}\prod_{j=2}^m Mf_j(x).
\end{equation*}
Hence, the $L^{p_1}\times L^{p_2}\times\cdots L^{p_m}\longrightarrow L^p$
boundedness follows easily for $1<p_1, p_2,\dots,p_m<\infty$ and
$1/p=1/p_1+1/p_2+\dots+1/p_m$.
For $\tilde{\mu}$, we have
\begin{theorem}\label{thm:existence-separated}
Theorems \ref{existence on BMO} and \ref{existence on Campanato space} are
still true for the operator $\tilde{\mu}$.
Theorem \ref{existence on Lipschitz space} also holds under the additional
condition $0<\alpha_j<1/m$, $j=1,\dots,m$.
\end{theorem}

\subsection{Proof of Theorem \ref{thm:existence-separated}}We only need
to show that:
For any $f_j\in\mathcal{E}^{\alpha_j,p_j}(\R^n)$, if there exists $x_0\in\R^n$
such that $\tilde{\mu}(\vec{f})(x_0))<\infty$, then for any ball $B\subset\R^n$
with $B\ni x_0$,
\begin{equation}\label{imp}
|\tilde{\mu}^\infty(\vec{f})(x_0)-\tilde{\mu}^\infty(\vec{f})(x)|
\le Cr^\alpha\prod_{j=1}^m
\|f_j\|_{\mathcal{E}^{\alpha_j,p_j}(\R^n)}.
\end{equation}
In fact, if the above inequality (\ref{imp}) holds, then
\begin{equation*}
\tilde{\mu}^\infty(\vec{f})(x)
\le \tilde{\mu}^\infty(\vec{f})(x_0)+Cr^\alpha\prod_{j=1}^m
\|f_j\|_{\mathcal{E}^{\alpha_j,p_j}(\R^n)}<+\infty.
\end{equation*}
Because of the $L^{p_1}\times L^{p_2}\times\cdots L^{p_m}\longrightarrow L^p$
boundedness, we have

\begin{equation*}
\biggl(\frac{1}{|B|}\int_B |\tilde{\mu}^r(\vec f)(x)|^p\,dx\biggr)^{1/p}
\le |B|^\alpha\prod_{j=1}^m
\|f_j\|_{\mathcal{E}^{\alpha_j,p_j}(\R^n)}<+\infty.
\end{equation*}
This implies that $\tilde{\mu}^r(\vec f)(x)<+\infty$ a.e. on $B$. Therefore,
 $\tilde{\mu}(\vec f)(x)<+\infty$ a.e. on $B$.

Since $B\ni x_0$ is arbitrary, it follows that
$\tilde{\mu}(\vec f)(x)<+\infty$ for almost every on $x\in \Rn$.
To show inequality (\ref{imp}), we follow the steps in the proof of
Theorems \ref{existence on BMO}-\ref{existence on Lipschitz space}.
As before, we see that
\begin{equation*}
\biggl|\int_{B(0,t)}\frac{\Omega_j({y_j})}{|{y_j}|^{n-1}} f_j(x-y_j)\; dy_j
\biggr|
\le Ct^{1+\alpha_j}\|f_j\|_{\mathcal E^{\alpha_j,1}}.
\end{equation*}
From this we see easily that the same estimates as \ref{eq:Ft-Campanato},
\ref{eq:B8r-i-Campanato}, and \ref{eq:B8r-Campanato} hold. Hence,
to show inequality (\ref{imp}) we have only to
show
\begin{equation*}
\int_{8r}^{\infty}|H_t(\vec{f})(x,z)|\;\frac{dt}{t^{m+1-\alpha}}
\le Cr^{2\alpha}.
\end{equation*}
where as before $H_t(\vec{f})(x,z)$ is defined by
\begin{align*}
&H_t(\vec{f})(x,z):=\biggl|
\int_{(B(z,t)\setminus B(z,8r))^m} \prod_{j=1}^m
\frac{\Omega_j(z-y_1)}{|z-y_j|^{n-1}}
 f_j(y_j)\; d\vec{y}\label{eq:Ht-f}
\\
&\hspace{4cm}-
\int_{(B(x,t)\setminus B(x,8r))^m} \prod_{j=1}^m
\frac{\Omega_j(x-y_1)}{|x-y_j|)^{n-1}}
 f_j(y_j)\; d\vec{y}\biggr|.\notag
\end{align*}
We may use the following estimate as before:
\begin{align*}
{}&H_t(\vec{f})(x,z)
\le\sum_{i=1}^5 H_{t,i}(\vec{f})(x,z)
+\sum_{l=1}^{m-1} H_{t,6}^l(\vec{f})(x,z)
+\sum_{l=1}^{m-1} H_{t,7}^l(\vec{f})(x,z) .
\end{align*}
\par\medskip\noindent
(I) As for $H_{t,1}(\vec f)(x,z)$, we get
\begin{align*}
|H_{t,1}(\vec f)(x,z)|
&\le
C\int_{\Xi(x,t)^m}\sum_{i=1}^{m}
\frac{r^\beta}{|x-y_i|^{n-1+\beta}}
\prod_{\ell\ne i}\frac{1}{|x-y_\ell|^{n-1}}
\prod_{j=1}^{m}|f_j(y)-(f_j)_B|\,d\vec{y}
\\
&\le C \sum_{j=1}^{m}\int_{8r\le|x-y_j|<t}
\frac{r^\beta|f_j(y)-(f_j)_B|}{|x-y_j|^{n-1+\beta}}\,dy_j
\prod_{\ell\ne j}\int_{8r\le|x-y_\ell|<t}
\frac{|f_\ell(y)-(f_\ell)_B|}{|x-y_\ell|^{n-1}}\,dy_\ell.
\end{align*}
(i) In the case $\alpha<0$, using Lemmas \ref{lem:BMO-n-1+beta} and
\ref{lem:BMO-n-1}, we have
\begin{equation*}
|H_{t,1}(\vec f)(x,z)|\le C\sum_{j=1}^{m}r^{\beta+\alpha_j}
t^{1-\beta}\prod_{\ell\ne j}(tr^{\alpha_\ell})
\prod_{j=1}^m\|f_j\|_{\mathcal {E}^{\alpha_j,1}}
=Cr^{\alpha+\beta}t^{m-\beta}\prod_{j=1}^m\|f_j\|_{\mathcal {E}^{\alpha_j,1}}.
\end{equation*}
So, we have
\begin{equation*}
\int_{8r}^{\infty}|H_{t,1}(\vec f)(x,z)|\frac{dt}{t^{m+1-\alpha}}
\le C r^{\beta+\alpha}\int_{8r}^{\infty}\frac{dt}{t^{1+\beta-\alpha}}
\prod_{j=1}^m\|f_j\|_{\mathcal {E}^{\alpha_j,1}}
\le Cr^{2\alpha}\prod_{j=1}^m\|f_j\|_{\mathcal {E}^{\alpha_j,1}}.
\end{equation*}

(ii) In the case $\alpha=0$, using Lemmas \ref{lem:BMO-n-1+beta} and
\ref{lem:BMO-n-1}, we have
\begin{equation*}
|H_{t,1}(\vec f)(x,z)|\le Cr^{\beta}t^{1-\beta}\Bigl(\log\frac{t}{r}\Bigr)
(t\log\frac{t}{r})^{m-1}
\prod_{j=1}^m\|f_j\|_{BMO}
=Cr^{\beta-\gamma}t^{m+\gamma-\beta}\prod_{j=1}^m\|f_j\|_{BMO},
\end{equation*}
for some $0<\gamma<\beta$.
So, we have
\begin{equation*}
\int_{8r}^{\infty}|H_{t,1}(\vec f)(x,z)|\frac{dt}{t^{m+1-\alpha}}
\le C r^{\beta-\gamma}\int_{8r}^{\infty}\frac{dt}{t^{1+\beta-\gamma}}
\prod_{j=1}^m\|f_j\|_{BMO}
\le C\prod_{j=1}^m\|f_j\|_{BMO}.
\end{equation*}

(iii) In the case $0<\alpha<\beta/2$,
using Lemmas \ref{lem:BMO-n-1+beta} and \ref{lem:BMO-n-1}, we have
\begin{equation*}
|H_{t,1}(\vec f)(x,z)|\le Cr^{\beta}t^{1-\beta+\alpha_j}
\prod_{\ell\ne j}t^{1+\alpha_j}
\prod_{j=1}^m\|f_j\|_{\mathcal {E}^{\alpha_j,1}}
=Cr^{\beta}t^{m+\alpha-\beta}\prod_{j=1}^m\|f_j\|_{\mathcal {E}^{\alpha_j,1}}.
\end{equation*}
So, we have
\begin{equation*}
\int_{8r}^{\infty}|H_{t,1}(\vec f)(x,z)|\frac{dt}{t^{m+1-\alpha}}
\le C r^{\beta}\int_{8r}^{\infty}\frac{dt}{t^{1+\beta-2\alpha}}
\prod_{j=1}^m\|f_j\|_{\mathcal {E}^{\alpha_j,1}}
\le Cr^{2\alpha}\prod_{j=1}^m\|f_j\|_{\mathcal {E}^{\alpha_j,1}}.
\end{equation*}
(II) As for $H_{t,2}(\vec{f})$, we have for $x,z\in B$,
\begin{align*}
&|H_{t,2}(\vec{f})(x,z)|
\leq C\int_{(B(x,10r)\backslash B(x,8r))^m}
\frac{\prod_{i=1}^m |f_i(y_i)-(f_i)_B|}{\prod_{i=1}^m|z-y_j|^{n-1}}
\;d\vec{y}
\\
&\leq C\prod_{i=1}^m \frac{1}{r^{n-1}}
\int_{B(x,10r)} |f_i(y_i)-(f_i)_B| \;dy_i
\\
&\leq Cr^{m+\alpha} \prod_{j=1}^m\|f_j\|_{\mathcal {E}^{\alpha_j,1}},
\end{align*}
which leads to
$$
\int_{8r}^\infty |H_{t,2}(\vec{f})(x,z)|\;\frac{dt}{t^{m+1-\alpha}}
\leq C r^{2\alpha}\prod_{j=1}^m\|f_j\|_{\mathcal {E}^{\alpha_j,1}},
$$
if $\alpha<m$. 
Similarly,
$$
\int_{8r}^\infty |H_{t,3}(\vec{f})(x,z)|\;\frac{dt}{t^{m+1-\alpha}}
\leq C r^{2\alpha}\prod_{j=1}^m\|f_j\|_{\mathcal {E}^{\alpha_j,1}}.
$$

\par\medskip\noindent
(III)
For $H_{t,4}(\vec{f})$, we have for any $x,z\in B$
$$
|\cap_{j=1}^m \Theta_j(x,y_j)|\leq |\Theta_i(x,y_i)|=|\Gamma(x,y_i)|
\leq ||z-y_i|-|x-y_i||\leq |z-x|\leq 2r,
$$
and for any $t\in \cap_{j=1}^m \Theta_j(x,y_j)$,
$t>\frac{1}{m}(\sum_{j=1}^m|x-y_j|)$.
 So, we obtain the following estimate:
\begin{align}\label{H_{t,4}}
&\int_{8r}^\infty |H_{t,4}(\vec{f})(x,z)|\;\frac{dt}{t^{m+1-\alpha}}
\\
&\leq C\int_{((B(x,8r))^c)^m}
\frac{\prod_{i=1}^m |f_i(y_i)-(f_i)_B|}{\prod_{i=1}^m|x-y_i|^{n-1}}
 \int_{\bigcap_{j=1}^m \Theta_j(x,y_j)} \frac{dt}{t^{m+1-\alpha}}\;d\vec{y}
 \notag
\\
&\leq C \int_{((B(x,8r))^c)^m}
\prod_{i=1}^m \frac{|f_i(y_i)-(f_i)_B|}{|x-y_i|^{n-1}}
\frac{r}{(\sum_{j=1}^m|x-y_j|)^{m+1-\alpha}} \;d\vec{y}.\nonumber
\\
&\le C\prod_{i=1}^m\int_{(B(x,8r))^c}
\frac{r^{1/m}|f_i(y_i)-(f_i)_B|}{|x-y_i|^{n+1/m-\alpha_i}}\,dy_i.\notag
\end{align}
We have a similar estimate for $H_{t,5}(\vec f)(x,z)$. Hence, using
Lemma \ref{lem:BMO-n+beta}, we get
\begin{align*}
&\int_{8r}^\infty |H_{t,4}(\vec{f})(x,z)|\;\frac{dt}{t^{m+1-\alpha}}
+\int_{8r}^\infty |H_{t,5}(\vec{f})(x,z)|\;\frac{dt}{t^{m+1-\alpha}}
\\
&\leq C\prod_{j=1}^m \int_{(B(x,8r))^c} \frac{r^{1/m}|f_j(y_j)-(f_j)_B|}
{|x-y_j|^{n+1/m-\alpha_j}}\;dy_j
+C\prod_{j=1}^m \int_{(B(z,8r))^c} \frac{r^{1/m}|f_j(y_j)-(f_j)_B|}
{|z-y_j|^{n+1/m-\alpha_j}}\;dy_j
\\
&\leq Cr^{2\alpha}\prod_{j=1}^m\|f_j\|_{\mathcal {E}^{\alpha_j,1}},
\end{align*}
if $\alpha_j<1/m$, $j=1,\dots,m$. 
\par\medskip\noindent
(IV) As for $H_{t,6}(\vec{f})$,
we see that
\begin{equation*}
|H_{t,6}^\ell(\vec{f})(x,z)|\le \prod_{j=1}^\ell
\int_{8r\le|x-y_j|<t}\frac{|f_j(y_j)-(f_j)_B|}{|x-y_j|^{n-1}}dy_j
\prod_{j=\ell+1}^m
\int_{8r\le|x-y_j|<10r}\frac{|f_j(y_j)-(f_j)_B|}{|x-y_j|^{n-1}}dy_j.
\end{equation*}
(i) In the case $\alpha<0$, we have by using Lemma \ref{lem:BMO-n-1}
\begin{align*}
|H_{t,6}^\ell(\vec{f})(x,z)|&\le
c\prod_{j=1}^\ell(t r^{\alpha_j})
\prod_{j=1}^\ell\|f_j\|_{\mathcal {E}^{\alpha_j,1}}
\prod_{j=\ell+1}^mr^{1+\alpha_j}
\prod_{j=\ell+1}^m\|f_j\|_{\mathcal {E}^{\alpha_j,1}}
=ct^\ell r^{m-\ell+\alpha}\prod_{j=1}^m\|f_j\|_{\mathcal {E}^{\alpha_j,1}}.
\end{align*}
Hence we have
\begin{equation*}
\int_{8r}^{\infty}|H_{t,6}^\ell(\vec{f})(x,z)|\frac{dt}{t^{m+1-\alpha}}
\le Cr^{m-\ell+\alpha}\int_{8r}^{\infty}\frac{dt}{t^{m-\ell+1-\alpha}}
\prod_{j=1}^m\|f_j\|_{\mathcal {E}^{\alpha_j,1}}
=Cr^{2\alpha}\prod_{j=1}^m\|f_j\|_{\mathcal {E}^{\alpha_j,1}}.
\end{equation*}
We have a similar estimate for $H_{t,7}^\ell(\vec{f})$.

(ii) In the case $\alpha=0$, we have by using Lemma \ref{lem:BMO-n-1}
\begin{align*}
|H_{t,6}^\ell(\vec{f})(x,z)|&\le
c\Bigl(t\log\frac{t}{r}\Bigr)^\ell\prod_{j=1}^\ell\|f_j\|_{BMO}
r^{m-\ell}\prod_{j=\ell+1}^m\|f_j\|_{BMO}
\le ct^{(1+\gamma)\ell}r^{m-(1+\gamma)\ell}\prod_{j=1}^m\|f_j\|_{BMO},
\end{align*}
for some sufficiently small $\gamma>0$.
Hence we have
\begin{equation*}
\int_{8r}^{\infty}|H_{t,6}^\ell(\vec{f})(x,z)|\frac{dt}{t^{m+1}}
\le Cr^{m-(1+\gamma)\ell}\int_{8r}^{\infty}
\frac{dt}{t^{m-(1+\gamma)\ell+1}}\prod_{j=1}^m\|f_j\|_{BMO}
=C\prod_{j=1}^m\|f_j\|_{BMO}.
\end{equation*}

(iii) In the case $\alpha>0$,
 we have by using Lemma \ref{lem:BMO-n-1}

\begin{align*}
|H_{t,6}^\ell(\vec{f})(x,z)|&\le
c\prod_{j=1}^\ell t^{1+\alpha_j}\|f_j\|_{\mathcal {E}^{\alpha_j,1}}
\prod_{j=\ell+1}^m r^{1+\alpha_j}\|f_j\|_{\mathcal {E}^{\alpha_j,1}}
=c t^{\ell+\sum_{j=1}^{\ell}\alpha_j} r^{m-\ell+\sum_{j=\ell+1}^m\alpha_j}
\prod_{j=1}^m\|f_j\|_{\mathcal {E}^{\alpha_j,1}}.
\end{align*}
Hence we have
\begin{align*}
\int_{8r}^{\infty}|H_{t,6}^\ell(\vec{f})(x,z)|\frac{dt}{t^{m+1-\alpha}}
&\le Cr^{m-\ell+\sum_{j=\ell+1}^m\alpha_j}
\int_{8r}^{\infty}\frac{dt}{t^{m-({\ell+\sum_{j=1}^{\ell}\alpha_j})+1-\alpha}}
\prod_{j=1}^m\|f_j\|_{\mathcal {E}^{\alpha_j,1}}
\\
&=Cr^{2m\alpha}\prod_{j=1}^m\|f_j\|_{\mathcal {E}^{\alpha_j,1}},
\end{align*}
if $\alpha-\alpha_j<1$, $j=1,\dots,m$, in particular,
if $\alpha_j<1/m$, $j=1,\dots,m$.  {\color{red} }

We have a similar estimate for $H_{t,7}^\ell(\vec{f})$.

This completes the proof of Theorem \ref{thm:existence-separated}.\qed
\subsection{More general type kernels}
\begin{definition}\label{def:mul-Mar-Int} (Multilinear Marcinkiewicz integral)
Let $\Omega$ be a function defined on $(\R^n)^m$ satisfying the conditions
(i), (ii) and (iii) in Definition \ref{def:mul-Mar-Int}.
For any $\vec{f}=(f_1,\cdots,f_m)\in \S(\Rn)\times \cdots \times \S(\Rn)$,
we define the operator $F_t$ for any $t>0$ as
\begin{align}
F_t(\vec{f})(x)&=\frac{1}{t^m}\int_{(B(0,t))^m} \frac{\Omega(\vec{y})}
{\prod_{i=1}^m|y_j|^{n-1}} \prod_{i=1}^m f_i(x-y_i)\; d\vec{y},
\label{F_t-moregeneral}
\end{align}
where $|\vec{y}|=|y_1|+\cdots+|y_m|$ and $B(x,t)=\{y\in \R^n:|y-x|\leq t\}$.
 Define the multilinear Marcinkiewicz integral $\mu$ by
\begin{align}\label{eq:Mar-Int}
\mu(\vec{f})(x)
=\left(\int_0^\infty |F_t(\vec{f})(x)|^2\;\frac{dt}{t}\right)^{1/2}.
\end{align}
\end{definition}
Then an open conjecture is as follows:
\begin{conjecture}
If $\Omega$ satisfies the conditions (i), (ii) and (iii) in Definition
\ref{def:mul-Mar-Int} and $\mu$ is bounded from
$L^{q_1}(\Rn)\times L^{q_2}(\Rn)\times \dots \times L^{q_m}(\Rn)$ to
$L^{q}(\Rn)$
for some $1<q_1,q_2,\dots,q_m<\infty$ with $1/q=1/q_1+1/q_2+\dots+1/q_m$, then, $\mu$ is bounded from
$L^{p_1}(\Rn)\times L^{p_2}(\Rn)\times \dots \times L^{p_m}(\Rn)$ to
$L^{p}(\Rn)$ for any $1<p_1,p_2,\dots,p_m<\infty$ with
$1/p=1/p_1+1/p_2+\dots+1/p_m$.
\end{conjecture}
We can only prove it for some special case, here we give two examples.
\begin{example}\label{ex:example-1}
\begin{equation*}
\Omega(x)=\prod_{j=1}^m \Omega_j(x_j),
\end{equation*}
where $\Omega_j$ is homogeneous of degree 0 on $\Rn$, Lipschitz continuous on
$S^{n-1}$, and $\int_{S^{n-1} }\Omega_j(y_j) \; dy_j=0$ for
$j=1,\dots,m$.
\end{example}
\begin{example}\label{ex:example-2}
\begin{equation*}
\Omega(x)=\sin\Bigl(\prod_{j=1}^m \Omega_j(x_j)\Bigr),
\end{equation*}
where $\Omega_j$ is odd and homogeneous of degree 0 on $\Rn$, and Lipschitz continuous on
$S^{n-1}$ for $j=1,\dots,m$.
\end{example}


{\bf Claim 1. }
For $1<p_1,p_2,\dots,p_m<\infty$ with $1/p=1/p_1+1/p_2+\dots+1/p_m$,
$\mu$ in Example \ref{ex:example-2} is bounde from
$L^{p_1}(\Rn)\times L^{p_2}(\Rn)\times \dots \times L^{p_m}(\Rn)$ to
$L^{p}(\Rn)$.
\begin{proof}
Since 
$\sin z=\sum_{k=0}^{\infty}(-1)^k\frac{z^{2k+1}}{(2k+1)!}$  for
$|z|<\infty$, we have
\begin{equation}
\sin\Bigl(\prod_{j=1}^m \Omega_j(x_j)\Bigr)
=\sum_{k=0}^{\infty}(-1)^k\frac{(\prod_{j=1}^m \Omega_j(x_j))^{2k+1}}{(2k+1)!}.
\end{equation}
Since the above convergence of $\sin z$ is uniform on every compact set of
$\C$, and $\Omega_j$'s are Lipschitz continuous on $\Sn$, this convergence is
uniform on $B(0,t)\times B(0,t)\times\dots\times B(0,t)$ for any fixed
$0\le t<\infty$. Hence, for any $f_j\in\S(\Rn)$ $(j=1,\dots,m)$, we have
\begin{align}
&\lim_{N\to\infty}\sum_{k=0}^{N}\frac{(-1)^k}{(2k+1)!}
\int_{(B(0,t))^m}\frac{(\prod_{j=1}^m \Omega_j(y_j))^{2k+1}}
{\prod_{j=1}^m|y_j|^{n-1}}\prod_{j=1}^m f_j(x-y_j)\,d\vec y
\\
&=\int_{(B(0,t))^m}
\frac{\Omega(y_1,y_2,\dots,y_m)}
{\prod_{j=1}^m|y_j|^{n-1}}\Pi_{j=1}^m f_j(x-y_j)\,d\vec y. \notag
\end{align}
So, by Fatou's Lemma we get
\begin{align}
&\mu(\vec f)(x)\label{eq:mul-mar-int}
\\
&=\biggl(\int_0^\infty\biggl|\frac{1}{t^m}\int_{(B(0,t))^m}
\frac{\Omega(y_1,y_2,\dots,y_m)}
{\prod_{j=1}^m|y_j|^{n-1}}\prod_{j=1}^m f_j(x-y_j)\,d\vec y\biggr|^2
\frac{dt}{t}\biggr)^{1/2}\notag
\\
&=\biggl(\int_0^\infty\lim_{N\to\infty}\biggl|\frac{1}{t^m}
\sum_{k=0}^{N}\frac{(-1)^k}{(2k+1)!}
\int_{(B(0,t))^m}\frac{(\prod_{j=1}^m \Omega_j(y_j))^{2k+1}}
{\prod_{j=1}^m|y_j|^{n-1}}\prod_{j=1}^m f_j(x-y_j)\,d\vec y\biggr|^2
\frac{dt}{t}\biggr)^{1/2}\notag
\\
&\le \liminf_{N\to\infty}\biggl(\int_0^\infty\biggl|\frac{1}{t^m}
\sum_{k=0}^{N}\frac{(-1)^k}{(2k+1)!}
\int_{(B(0,t))^m}\frac{(\prod_{j=1}^m \Omega_j(y_j))^{2k+1}}
{\prod_{j=1}^m|y_j|^{n-1}}\prod_{j=1}^m f_j(x-y_j)\,d\vec y\biggr|^2
\frac{dt}{t}\biggr)^{1/2} \notag
\\
&\le \liminf_{N\to\infty}\sum_{k=0}^{N}\frac{(-1)^k}{(2k+1)!}
\biggl(\int_0^\infty\biggl|\frac{1}{t^m}
\int_{(B(0,t))^m}\frac{(\prod_{j=1}^m \Omega_j(y_j))^{2k+1}}
{\prod_{j=1}^m|y_j|^{n-1}}\prod_{j=1}^m f_j(x-y_j)\,d\vec y\biggr|^2
\frac{dt}{t}\biggr)^{1/2} \notag
\\
&=\sum_{k=0}^{\infty}\frac{(-1)^k}{(2k+1)!}
\biggl(\int_0^\infty\biggl|\prod_{j=1}^m \frac{1}{t}
\int_{B(0,t)}\frac{(\Omega_j(y_j))^{2k+1}}
{|y_j|^{n-1}}f_j(x-y_j)\,dy_j\biggr|^2\frac{dt}{t}\biggr)^{1/2}. \notag
\end{align}
We see that
\begin{equation*}
\biggl|\frac{1}{t}
\int_{B(0,t)}\frac{(\Omega_j(y_j))^{2k+1}}
{|y_j|^{n-1}}f_j(x-y_j)\,dy_j\biggr|
\le C\|\Omega\|_\infty^{2k+1}Mf_j(x).
\end{equation*}
So, we have
\begin{align}
&\biggl(\int_0^\infty\biggl|\prod_{j=1}^m \frac{1}{t}
\int_{B(0,t)}\frac{(\Omega_j(y_j))^{2k+1}}
{|y_j|^{n-1}}f_j(x-y_j)\,dy_j\biggr|^2\frac{dt}{t}\biggr)^{1/2}
\label{eq:max-mar}
\\
&\le C\|\Omega\|_\infty^{2k+1}\prod_{j=2}^mMf_j(x)\biggl(\int_0^\infty
\biggl|\frac{1}{t}\int_{B(0,t)}\frac{(\Omega_1(y_1))^{2k+1}}
{|y_1|^{n-1}}f_1(x-y_1)\,dy_1\biggr|^2\frac{dt}{t}\biggr)^{1/2},\notag
\end{align}
We see that
\begin{equation*}
\|(\Omega_1(y_1))^{2k+1}\|_{\mathrm{Lip}_\alpha(\Sn)}\le Ck
\|\Omega_1\|_\infty^{2k+1}\|\Omega_1\|_{\mathrm{Lip}_\alpha(\Sn)},
\end{equation*}
which implies that for every $1<q<\infty$ there exist some $A>0$ and $C_q>0$
 such that
\begin{equation}\label{eq:mu-1-1}
\biggl\|\biggl(\int_0^\infty
\biggl|\frac{1}{t}\int_{B(0,t)}\frac{(\Omega_1(y_1))^{2k+1}}
{|y_1|^{n-1}}f_1(x-y_1)\,dy_1\biggr|^2\frac{dt}{t}\biggr)^{1/2}
\biggr\|_{L^q(\Rn)}
\le C_qkA^k\|f_1\|_{L^q(\Rn)}.
\end{equation}
By \eqref{eq:mul-mar-int}, \eqref{eq:max-mar}, \eqref{eq:mu-1-1},
the $L^p$ boundedness of the Hardy-Littlewood maximal function,
Minkowski's inequality and Young's inequality, we obtain
for $1<p_1,p_2,\dots,p_m<\infty$ with $1/p=1/p_1+1/p_2+\dots+1/p_m$
\begin{equation}
\|\mu(\vec f)\|_{L^p(\Rn)}
\le C\sum_{k=0}^{\infty}\frac{(-1)^k}{(2k+1)!}kA^k\prod_{j=1}^{m}
\|f_j\|_{L^{p_j}(\Rn)}.
\end{equation}
which completes the proof of our Claim 1.

\end{proof}


\end{document}